\documentclass{amsart}
\usepackage{amsmath}
\usepackage{amsfonts}
\usepackage{amssymb}
\usepackage{epsfig}
\newtheorem{theorem}{Theorem}[section]

\newtheorem{lm}[theorem]{Lemma}
\newtheorem{tr}[theorem]{Theorem}
\newtheorem{cor}[theorem]{Corollary}

\newtheorem{rem}[theorem]{Remark}
\newtheorem{pr}[theorem]{Proposition}
\usepackage{color}

\begin{document}

\title{The center of $Dist(GL(m|n))$ in positive characteristic}
\thanks{This publication was made possible by a NPRF award NPRP 6 - 1059 - 1 - 208 from the Qatar National Research Fund 
(a member of The Qatar Foundation). The statements made herein are solely the responsibility of the authors.}
\author{Alexandr N. Zubkov}
\address{Omsk State Pedagogical University \\ Chair of Geometry \\ 644099 Omsk-99 \\ Tuhachev-skogo Embankment 14 \\ Russia}
\email{a.zubkov@yahoo.com}
\author{Franti\v sek Marko}
\address{Pennsylvania State University \\ 76 University Drive \\ Hazleton, PA 18202 \\ USA}
\email{fxm13@psu.edu}
\begin{abstract}
The purpose of this paper is to investigate central elements in distribution algebras $Dist(G)$ of general linear supergroups $G=GL(m|n)$. As an application, we compute explicitly the center of $Dist(GL(1|1))$ and its image under Harish-Chandra homomorphism.
\end{abstract}
\maketitle

\section*{Introduction}
The main motivation of this article was the paper \cite{hab} of Haboush who related the adjoint invariants to central elements of the distribution algebra of a semisimple simply connected algebraic group over the field of positive characteristic. Our main goal was to extend his results and method to the case of supergroups.

Let $G$ be an algebraic supergroup and $R$ be its normal infinitesimal subsupergroup. Then $Dist(R)$ has a right integral $\nu$ which is semi-invariant with respect to the adjoint action of $G$ on $Dist(R)$ corresponding to a character $\chi$. We show that the map $K[R]^{G ,-\chi} \to Dist(R)^G$ given by $r\mapsto \nu r$ is an isomorphism of superspaces.

Let $G_{ev}$ be the largest even subsupergroup of $G$ and $G_r$ be the $r$-th Frobenius kernel of $G$.
In the special case, when the above character is trivial, we have an isomorphism 
$K[G_r]^G \to Dist(G_r)^G$. For example, this is valid for the general linear, special linear and ortho-symplectic supergroups.

Additionally, if $Dist(G)$ is generated by $Dist(G_{ev})$ and some odd primitives, then $Dist(G)^G$ is the center of $Dist(G)$. Since $Dist(G)^G=\cup_{r\geq 1} Dist(G_r)^G$, this gives a description of the center of $Dist(G)$.

For $G=GL(m|n)$ and its Frobenius kernel $G_r$, we can extend the integral on $Dist(G_{ev,r})$ to an integral on $Dist(G_r)$ using the isomorphism of affine superschemes
$G_r\simeq V^-_r\times G_{ev,r}\times V^+_r$, where $V^-_r$ and $V^+_r$ are odd unipotent subsupergroups of $G_r$ corresponding to the lower odd block and upper odd block of $G_r$, respectively. We expect that this method will work for other supergroups.

In Section 1, we recall properties of Hopf superalgbras and their right and left integrals, and in Section 2 we specialize to finite-dimensional Hopf superalgebras.
In Sections 3 and 4 we recall basic definitions and properties of algebraic supergroups and their superalgebras of distributions.
Section 5 is devoted to auxiliary results about central elements and adjoint invariants for superalgebras $A$ and distribution superalgebras $Dist(G)$.
In Section 6, we follow the method of Haboush and describe central elements and integrals in $Dist(GL(m|n))$. Section 7 is devoted to Harish-Chandra homomorphism
from the center of $Dist(G)$ to the distribution superalgebra $Dist(T)$ of a maximal torus $T$ of $G$, and formulas for the conjugation action of odd unipotent subsupergroup $U_{ij}$ of $G$ on $K[G]$ and on $Dist(G)$. Such actions are completely determined by the superderivation 
$D_{ij}$.
In Section 8, we compute explicitly the invariants of $K[G_r]^G$ for $G=GL(1|1)$ and the center of $Dist(G)$.
We finish in section 9 with a discussion of Kujawa and Harish-Chandra blocks and their relationship to blocks of the supergroup $G$.

\section{Hopf superalgebras}

A vector superspace $V$ is a $\mathbb{Z}_2$-graded vector space, say $V=V_0\bigoplus V_1$. If $v\in V_i$ for $i=0, 1,$ then
$i$ is called a {\it parity} of $v$ and it is denoted by $|v|$.
The vector superspaces form 
a symmetric tensor category with respect to the (super)symmetry
$$t_V : v\otimes w\mapsto (-1)^{|v||w|}w\otimes v \mbox{ for } v, w\in V .$$
If we admit ungraded morphisms between vector superspaces, then this category is not abelian, but its {\it underlying even category},  consisting of the same objects but only even morphisms, is. For any superspace $V$ we define 
a {\it parity shift} $\Pi V$ of $V$ by $(\Pi V)_i=V_{i+1}$ for $i\in\mathbb{Z}_2$. The functor $V\to \Pi V$ is an auto-equivalence of the category of vector superspaces.

Algebraic systems, such as algebra or Hopf (or Lie)
algebra, defined in the tensor category of vector superspaces, 
will be labelled with the prefix "super." For example, a {\it superalgebra} is just associative $\mathbb{Z}_2$-graded algebra. However, the tensor product $A\otimes B$ of two superalgebras involves the supersymmetry; the product in $A\otimes B$ is defined as 
$$(a\otimes b)(c\otimes d)=(-1)^{|b||c|}ac\otimes bd \mbox{ for } a, c\in A \mbox{ and } b, d\in B.$$
A superalgebra $A$ is said to be {\it supercommutative}, if it satisfies the identity
$$ab=(-1)^{|a||b|}ba \mbox{ for each } a, b\in A.$$
The categories of left/right $A$-supermodules with ungraded morphisms are denoted by ${}_{A}\mathsf{SMod}$ and
$\mathsf{SMod}_A$ respectively. Their underlying even categories are denoted by ${}_{A}\mathsf{Smod}$ and
$\mathsf{Smod}_A$ respectively.

If $A$ is a {\it supercoalgebra}, then the coproduct $\Delta_A :
A\to A\otimes A$ is required to be an even morphism. Furthermore, if $V$ is a left {\it supercomodule} over $A$, then its comodule
map $\tau_V : V\to A\otimes V$ is required to be even morphism as well. 
The categories of left/right $A$-supercomodules with ungraded morphisms are denoted by ${}^{A}\mathsf{SMod}$ and
$\mathsf{SMod}^A$ respectively. 
Their underlying even categories are denoted by ${}^{A}\mathsf{Smod}$ and $\mathsf{Smod}^A$ respectively. 

If $A$ is a Hopf superalgebra, then the coproduct $\Delta_A$, the antipode $s_A$ and the counit $\epsilon_A$ are superalgebra morphisms.  
For a subset $V\subset A$, we denote by $V^+=V\cap\ker\epsilon_A$.
A Hopf superalgebra $A$ has two right $A$-supercomodule structures, denoted by $A_r$ and $A_l$, given by $\tau_{A_r}=\Delta_A$ and $\tau_{A_l}=t_A(s_A\otimes id_A)\Delta_A$ respectively. 
The antipode $s_A$ maps $A_r$ isomorphically to $A_l$.

If $V\in {}^{A}\mathsf{SMod}$, then one can define the {\it invariant subsuperspace} as ${}^{A}V=\{v\in V|\tau_V(v)=1\otimes v\}$. Symmetrically, if $V\in \mathsf{SMod}^A$, then $V^A=\{v\in V| \tau_V(v)=v\otimes 1\}$.

Let $V$ be a finite-dimensional right $A$-supercomodule (the case of finite-dimen-sional left $A$-supercomodule is analogous). Then the dual space $V^*$ has a right $A$-supercomodule structure given by 
$\tau_{V^*}(\phi)=\sum\phi_1\otimes a_2$, where 
\[\sum(-1)^{|a_2||v_1|}\phi_1(v_1)a_2 a'_2=\phi(v)\]
for every $v\in V$ and $\tau_V(v)=\sum v_1\otimes a'_2$. More precisely, 
choose a homogeneous basis consisting of elements $v_i$ for $1\leq i\leq t$.  If $\tau_V(v_i)=\sum_{1\leq k\leq t} v_k\otimes a_{ki}$ for $1\leq i\leq t$, then $A$-supercomodule $V^*$ has a (dual) basis $v^*_i$ such that
$$\tau_{V^*}(v^*_i)=\sum_{1\leq k\leq t} v^*_k \otimes (-1)^{|v_k|(|v_i|+|v_k|)}s_A(a_{ik}).$$ 
The functor $V\to V^*$ is an anti-equivalence of the full subcategory of $\mathsf{SMod}^A$ that consists of all finite-dimensional $A$-supercomodules. We will use the notation $\tau^*_V$ for $\tau_{V^*}$.

Assume that $V$ is a right $A$-supercomodule and a right $A$-supermodule. Then $V$ is called a right {\it Hopf supermodule} if for all $v\in V$ and $a'\in A$
we define
\[\tau_V(va')=\sum (-1)^{|a_2||a'_1|}v_1a'_1\otimes a_2 a'_2,\] 
where $\tau_V(v)=\sum v_1\otimes a_2$ and $\Delta_A(a')=\sum a'_1\otimes a'_2$. A left $A$-Hopf supermodule is defined analogously. 

If $V$ is a left or right $A$-Hopf supermodule, then so is $\Pi V$.

Typical examples of right $A$-Hopf supermodules are $A_l$ and $A_r$, where the right $A$-supermodule structure of $A_r$ is natural but
the right $A$-supermodule structure of $A_l$ is modified as $a\star b=as_A(b)$ for $a, b\in A$.
If $A$ is finite-dimensional, then both $A_r^*$ and $A_l^*$ are right $A$-Hopf supermodules with respect to the actions $(x a)(a')=x(aa')$ 
and $(x\star a)(a')=x(s_A(a)a')$, for $x\in A^*$ and $a, a'\in A$. 
\begin{lm}\label{larsentheorem}
If $V$ is a right $A$-Hopf supermodule, then there is an isomorphism of $A$-Hopf supermodules $V^A\otimes A\to V$ given by
$v\otimes a\mapsto va$ for $v\in V^A$ and $a\in A$. In particular, every right $A$-Hopf supermodule is isomorphic to
$A_r^{\oplus\dim V^A_0}\oplus (\Pi A_r)^{\oplus\dim V^A_1}$. Both statements are valid also for left $A$-Hopf supermodules.
\end{lm}
\begin{proof}
The proof can be modified from Theorem 4.1.1 of \cite{sw}.
\end{proof}

Define the subsuperspaces of {\it right} and {\it left integrals} on $A$ by
\[\int_{r, A}=\{\nu\in A | \nu a=\epsilon_{A}(a)\nu \mbox{ for all } a\in A\} \]
and
\[\int_{l, A}=\{\nu\in A | a\nu=\epsilon_{A}(a)\nu \mbox{ for all } a\in A\}.\]

Define the left {\it adjoint} action of $A$ on itself by
\[{\bf ad}(x)y=\sum (-1)^{|x_2||y|}x_1y s_A(x_2),\]
for $x, y\in A$. Symmetrically, define the right {\it adjoint}
action by 
\[y{\bf ad}(x)=\sum (-1)^{|y||x_1|}s_A(x_1)yx_2.\]
A Hopf supersubalgebra $B\subseteq A$ is called {\it left normal}
(or {\it right normal}, respectively) if ${\bf ad}(x)y\in B$ for all $x\in A, y\in B$ (or $y{\bf ad}(x)\in B$ for all $x\in A, y\in B$, respectively). If $A$ is supercocommutative, then ${\bf ad}(s_A(x))y=(-1)^{|y||x|}y{\bf ad}(x)$. In particular, if $s_A$ is a bijection, then the left normality of $B$ in $A$ is equivalent to the right normality of $B$ in $A$.
\begin{lm}\label{integralsareinvariant}
Assume that $A$ is a supercocommutative Hopf superalgebra with bijective antipode and $B$ is normal in $A$. Then the supersubspaces 
$\int_{l,B}$ and $\int_{r, B}$ are invariant with respect to both adjoint actions.  
\end{lm}
\begin{proof}
Let us consider $y\in\int_{r, B}$. For every $x\in A$ and $z\in B$ we obtain
\[
\begin{aligned}
({\bf ad}(x)y)z=&\sum (-1)^{(|x_2|+|x_3|+|x_4|)|y|+(|x_3|+|x_4|)|z|}
x_1ys_A(x_2)zx_3 s_A(x_4)\\
=&\sum (-1)^{(|x_2|+|x_3|)(|y|+|z|)}x_1 (y z{\bf ad}(x_2))s_A(x_3)\\
=&\sum (-1)^{(|x_2|+|x_3|)(|y|+|z|)}x_1\epsilon_A(z{\bf ad}(x_2))y s_A(x_3)=\epsilon_A(z){\bf ad}(x)y.
\end{aligned}\]
The proof of the remaining cases is similar.
\end{proof}

\section{Finite-dimensional Hopf superalgebras}

In what follows, we will use the definitions and notations from \cite{maszub, zub1}. Let $R$ be a finite-dimensional Hopf superalgebra. The dual superspace $R^*$ has a natural structure of a Hopf superalgebra. The superalgebra structure on $R^*$ is given by
\[(\phi\psi)(r)=\sum (-1)^{|\psi||r_1|}\phi(r_1)\psi(r_2),\]
where $\phi, \psi\in R^*$ and $r\in R.$ 
The comultiplication is given by $\Delta_{R^*}(\phi)=\sum\phi_1\otimes\phi_2$, where 
$\phi(rs)=\sum(-1)^{|\phi_2||r|}\phi_1(r)\phi_2(s)$ for arbitrary $r, s\in R$. 
The antipode and counit are given as $s_{R^*}(\phi)(r)=\phi(s_R(r))$ and $\epsilon_{R^*}(\phi)=\phi(1_R)$.
\begin{lm}\label{self-duality}
The functor $R\to R^*$ is a self-duality on the category of finite-dimensional Hopf superalgebras.  
\end{lm}
\begin{proof}
Routine verification analogous to I.8(1) of \cite{jan}.
\end{proof}
If $M$ is a left $R$-supermodule, then $M$ is a right $R^*$-supercomodule via the linear map
$M\to Hom_K(R, M)\simeq M\otimes R^*$ which maps $m\in M$ to $l_m : R\to M$ given by $l_m(r)=(-1)^{|m||r|}rm$ for $r\in R$.
On the other hand, if $M$ is a right $R$-supercomodule, then $M$ is a left $R^*$-supermodule by 
$\phi m=\sum(-1)^{|\phi||m_1|}\phi(r_2)m_1,$ where $\tau_M(m)=\sum m_1\otimes r_2,$ for $\phi\in R^*$ and $m\in M$.
If we replace left action by right action and right coaction by left coaction in the above statements, they will remain valid.

The proof of the following lemma is easy and we leave it to the reader.
\begin{lm}\label{equivalence}
There is an equivalence of categories ${}_{R}\mathsf{SMod}$ of left $R$-supermodules and $\mathsf{SMod}^{R^*}$ of right $R^*$-supercomodules.
Analogously, there is an equivalence of categories $\mathsf{SMod}_R$ of right $R$-supermodules and ${}^{R^*}\mathsf{SMod}$ of left $R^*$-supercomodules.
\end{lm}
As a consequence, the categories $\mathsf{SMod}^R$ and ${}^{R}\mathsf{SMod}$ have enough projective objects.

Lemma \ref{larsentheorem} implies that $(R^*_s)^R$ is a one-dimensional even or odd superspace for  $s\in\{r, l\}$.
It is easy to see that elements from $(R^*_r)^R$ are right integrals. In other words, we have 
\[\int_{r, R^*}=(R^*_r)^R=\{\nu\in R^* | \nu x=\epsilon_{R^*}(x)\nu=x(1)\nu \mbox{ for all } x\in R^*\}.\]
For example, $\nu\in (R^*_r)^R$ if and only if $\sum_{|h_1|=|\nu|}\nu(h_1)h_2=\nu(h)$ for all $h\in R$. The latter is equivalent to
\[\nu\mu(h)=\sum_{|h_1|=|\nu|}\nu(h_1)\mu(h_2)=\mu(1)\nu(h)\]
for every $\mu\in R^*$. 
\begin{cor}\label{rightisom}
If $\nu\in\int_{r, R^*}\setminus 0$, then $r_{\nu} : \Pi^{|\nu|}R_r\to R_r^*$ given by $h\mapsto \nu h$ is an isomorphism of
right $R$-Hopf supermodules. 
\end{cor}
\begin{rem}\label{otherisom}
Consider the diagram 
\[\begin{array}{ccc}
\Pi^{|\nu|}R_r & \stackrel{\sim}{\to} & R^*_r \\
\uparrow & & \downarrow \\
\Pi^{|\nu|}R_l & & R_l^*
\end{array},
\] 
where the vertical arrows are supercomodule isomorphisms given by $h\mapsto s_R(h), x\mapsto s_{R^*}(x)=x s_R$ for $h\in R, x\in R^*$, and the upper
arrow is the isomorphism from Corollary \ref{rightisom}. The composition of these morphisms gives a Hopf supermodule isomorphism $l_{\mu} : \Pi^{|\nu|}R_l\to R_l^*$ 
given by $h\mapsto \mu h$, where $\mu=s_{R^*}(\nu)\in (R^*_l)^R$. 
\end{rem}

Working on the left, one can show that $\int_{R^*, l}$ is a one-dimensional even or odd superspace. 
Lemma \ref{self-duality} implies that every finite-dimensional Hopf superalgebra has non-zero left and right integrals.

\section{Algebraic supergroups}

Let $\mathsf{SAlg}_K$ denote the category of supercommutative $K$-superalgebras with even morphisms. Let $B$ be a supercommutative Hopf superalgebra. Then the representable functor $A\to Hom_{\mathsf{SAlg}_K}(B, A)$ from $\mathsf{SAlg}_K$ to the category of sets is a natural group 
functor. It is denoted by $G=SSp \ B$ and called an {\it affine supergroup}. If $B$ is finitely generated, then $G$ is called
an {\it algebraic supergroup}. Any (closed) subsupergroup $H$ of $G$ is uniquely defined by a Hopf superideal $I_H$
such that an element $g\in G(A)$ belongs to $H(A)$ if and only if $g(I_H)=0$. For example, the {\it largest even subsupergroup} 
$G_{ev}$ of $G$ is defined by the ideal $BB_1$. 

Let $W$ be a finite-dimensional superspace. The group functor $A\to \mathsf{End}_A(W\otimes A)_{0}^*$ is an algebraic supergroup. It is called a {\it general linear supergroup} and denoted by $GL(W)$. If $\dim W_0=m$ and $\dim W_1=n$, then $GL(W)$ is also denoted by $GL(m|n)$. 
Fix a homogeneous basis of $W$, say $w_i$ for $1\leq i\leq m+n$, where $|w_i|=0$ for $1\leq i\leq m$, and
$|w_i|=1$ for $m+1\leq i \leq m+n$.

A generic matrix $C=(c_{ij})_{1\leq i, j\leq m+n}$, where $|c_{ij}|=|w_i|+|w_j|$, has a block form 
$$\left(\begin{array}{cc}
C_{00} & C_{01} \\
C_{10} & C_{11}
\end{array}\right)
$$
with even $m\times m$ and $n\times n$ blocks $C_{00}$ and $C_{11}$, and odd $m\times n$ and $n\times m$ blocks
$C_{01}$ and $C_{10}$ respectively. 
Then $K[GL(m|n)]=K[c_{ij}|1\leq i, j\leq m+n]_{d}$, where $d=\det(C_{00})\det(C_{11})$,  
$$\Delta_{GL(m|n)}(c_{ij})=\sum_{1\leq k\leq m+n} c_{ik}\otimes c_{kj} \mbox{ and } \ \epsilon_{GL(m|n)}(c_{ij})=\delta_{ij}.$$

By the definition, the category of left/right $G$-supermodules, denoted by $G-\mathsf{SMod}$ and $\mathsf{SMod}-G$ respectively, coincides with the category of right/left $K[G]$-supercomodules. 
Their even underlying categories with the same objects, but only graded (even) homomorphisms, are denoted by $G-\mathsf{Smod}$ and $\mathsf{Smod}-G$ respectively. 

For example, the right supercomodule structure of a $GL(W)$-supermodule $W$ is defined by
$$\tau_W(w_i)=\sum_{1\leq j\leq m+n} w_j\otimes c_{ji}.$$

There is a one-to-one correspondence between $G$-supermodule structures on a finite-dimensional superspace $V$ and {\it linear representations}
$G\to GL(V)$ (cf. \cite{jan, zub1}). If $\tau_V(v)=\sum v_1\otimes f_2\in V\otimes K[G]$, then $g\in G(A)$ acts on $V\otimes A$ by even $A$-linear automorphism $g(v\otimes a)=\sum v_1\otimes g(f_2)a$. 
In other words, $V$ is a $G$-supermodule if and only if the group functor $G$ acts on the functor $V_a=V\otimes ?$ in such a way that, for every $A\in\mathsf{SAlg}_K$, the group $G(A)$ acts on $V_a(A)=V\otimes A$ by even $A$-linear automorphisms. 
If $V$ is finite-dimensional, then for any $A\in\mathsf{SAlg}_K$ one can define a pairing
$(V^*\otimes A)\times (V\otimes A)\to A$ such that
\[<\phi\otimes a , v\otimes b>=(\phi\otimes a)(v\otimes b)=(-1)^{|a||v|}\phi(v)ab, \phi\in V^* \mbox{ for } v\in V \mbox{ and } a, b\in A.
\]
Then for every $g\in G(A)$ we have
$$(g(\phi\otimes a))(v\otimes b)=(\phi\otimes a)(g^{-1}(v\otimes b)).
$$
Observe also that $v$ belongs to $V^G$ if and only if $g(v\otimes 1)=v\otimes 1$ for every $g\in G(A)$ and $A\in\mathsf{SAlg}_K$.

The superalgebra $K[G]$ is a left $G$-supermodule with respect to the {\it conjugation} action. The corresponding right supercomodule structure  
is given by \[\mathfrak{con} : f\mapsto \sum (-1)^{|f_1||f_2|}f_2\otimes s_H(f_1) f_3\] that is induced by the right action of $G$ on itself by conjugations (cf. \cite{zub1}). Thus $K[G]$ is a left $G$-supermodule such that an element
$g\in G(C)$ acts on $K[G]\otimes C$ by the rule
\[g(f\otimes c)=\sum (-1)^{|f_1||f_2|}f_2\otimes g(s_H(f_1) f_3)c,\]
where $f\in K[G], c\in C$ and $C\in \mathsf{SAlg}_K.$
Since $\mathfrak{con}$ is a superalgebra homomorphism, this $C$-linear map corresponding to $g$ is a superalgebra automorphism on $K[G]\otimes C$.

Recall that $H$ is a normal subsupergroup of $G$
if and only if $\mathfrak{con}(I_H)\subseteq I_H\otimes K[G]$ (cf. p. 731 of \cite{zub1}). Therefore, if $H$ is a normal subgroup of $G$, then the map $\mathfrak{con}$ induces a $G$-supermodule structure on $K[H]$. 

As in the purely even case, there is an obvious bijection between one-dimensional $G$-supermodules and group-like elements of $K[G]$.
The latter ones form a subgroup $X(G)$ of the group $K[G]^{\times}$ consisting of all invertible elements of $K[G]$. We call $X(G)$ a {\it character group} of $G$. 

For example, if $G=GL(m|n)$, then the element $$Ber(C)=\det(C_{00}-C_{01}C_{11}^{-1}C_{10})\det(C_{11})^{-1}$$ is 
a group-like element of the Hopf superalgebra $K[GL(m|n)]$ (cf. \cite{ber}), called {\it Berezinian}.
One can also show that
$K[GL(m|n)]=K[c_{ij}|1\leq i, j\leq m+n]_{Ber(C)}$.

Let $B^+$ denote the Borel subsupergroup of $G$ that consists of all upper triangular matrices. The opposite Borel subsupergroup is denoted by $B^-$ and it consists of all lower triangular matrices. Their largest unipotent subsupergroups are denoted by
$U^+$ and $U^-$ respectively. They are unipotent radicals of $B^+$ and $B^-$ respectively (see \cite{amas2, zub3, zub4} for the definition of the unipotent radical in the category of supergroups). 

Let $T$ be a maximal torus in $G$ consisting of diagonal matrices. The group $X(T)$ can be naturally identified with the (additive) group $\mathbb{Z}^{m+n}$ in such a way that
$\lambda=(\lambda_1 ,\ldots , \lambda_{m+n})\in \mathbb{Z}^{m+n}$ corresponds to the group-like element $c^{\lambda}=\prod_{1\leq i\leq m+n}c_{ii}^{\lambda_i}$, where $K[T]=K[c_{11}^{\pm 1},\ldots , c_{m+n}^{\pm 1}]$. For each $i=1, \ldots, m+n$ denote by $\epsilon_i$ an element of $\mathbb{Z}^{m+n}$ that has all entries equal to zero, except the $i$-th entry that equals 1. 

The supergroups $B^+$ and $B^-$ are isomorphic to the semi-direct products $U^-\rtimes T$ and $U^+\rtimes T$, respectively.
Similarly to the purely even case, any irreducible $B^-$-supermodule is one-dimensional and it is uniquely determined by its highest weight $\lambda\in X(T)$ and by its parity. In what follows, we denote such a supermodule by $K^{a}_{\lambda}$ for $a\in\mathbb{Z}_2$. Then the induced supermodule
$H^0(\lambda^a)=ind^G_{B^-} K^{a}_{\lambda}$ is non-zero if and only if $\lambda$ is dominant, that is,
$\lambda_1\geq\ldots\geq\lambda_m$ and $\lambda_{m+1}\geq\ldots\geq\lambda_{m+n}$. 
Moreover, if $H^0(\lambda^a)\neq 0$, then its socle is an irreducible 
$G$-supermodule, denoted by  $L(\lambda^a)$, and each irreducible $G$-supermodule is isomorphic to some $L(\lambda^a)$. Also,  
$\Pi H^0(\lambda^0)=H^0(\lambda^1)$ and $\Pi L(\lambda^0)=L(\lambda^1)$. The set of dominant weights of $G$ will be denoted by $X(T)^+$.

An algebraic supergroup $G$ is said to be {\it finite}, if $K[G]$ is finite-dimensional. The supergroup $G$ is called {\it infinitesimal} if
$\mathfrak{m}=K[G]^+$ is nilpotent. Any infinitesimal supergroup is obviously finite (cf. \cite{zub1}).

The $r$-th {\it Frobenius kernel} $G_r$ of an algebraic supergroup $G$ is defined by the Hopf superideal $\sum_{f\in \mathfrak{m}_0}K[G]f^{p^r}$.
It is clear that $G_r$ is a normal and infinitesimal subsupergroup of $G$.

Extending the notion of one-dimensional unipotent group $G_a$ (that can be also regarded as a purely even supergroup), one can define
{\it one-dimensional odd unipotent} supergroup $G_a^-$. More precisely, $A=K[G_a^-]=K[x]$, where $|x|=1$, $\Delta_A (x)=x\otimes 1+1\otimes x$, $s_A(x)=-x$ and 
$\epsilon_A(x)=0$. Each subsupergroup $U_{ij}$ of $GL(m|n)$, that is defined by the Hopf superideal
generated by elements $c_{kl}-\delta_{kl}$, where $k\neq i$ or $l\neq j$ and $|c_{ij}|=1$, is isomorphic to the supergroup $G_a^-$.

\section{Superalgebras of distributions}

Let $G$ be an algebraic supergroup. Then
$Dist(G)$ denotes  the {\it superalgebra of distributions} of $G$. As a superspace, $Dist(G)$ coincides with
$\bigcup_{k\geq 0} Dist_k(G)\subseteq K[G]^*$, where $Dist_k(G)=(K[G]/\mathfrak{m}^{k+1})^*$. $Dist(G)$ is a supercocommutative Hopf superalgebra with bijective antipode (see \cite{jan, zub1} for more details). For example, the comultiplication $\Delta_{Dist(G)}$ maps an element $\phi\in Dist_k(G)$
to $\sum \phi_1\otimes\phi_2\in Dist_k(G)^{\otimes 2}$, where
$$\phi(fh)=\sum (-1)^{|\phi_2||f|}\phi_1(f)\phi_2(h)$$ for every $f, h\in K[G]$.

If $H$ is a subsupergroup of $G$, then $Dist(H)$ is a Hopf subsuperalgebra of $Dist(G)$. More precisely,
$\phi\in Dist(G)$ belongs to $Dist(H)$ if and only if $\phi(I_H)=0$.

Every left $G$-supermodule $V$ (the case of a right supermodule is symmetrical) is also a left $Dist(G)$-supermodule using the action 
$\phi\cdot v=\sum (-1)^{|\phi||v_1|}v_1\phi(f_2)$, where $\phi\in Dist(G), v\in V$ and $\tau_V(v)=\sum v_1\otimes f_2$. If $V$ is finite-dimensional,
then $Dist(G)$ acts on $V^*$ by the rule $(\phi\cdot\psi)(v)=(-1)^{|\phi||\psi|}\psi(s_{Dist(G)}(\phi)\cdot v)$ for $v\in V, \psi\in V^*$ and $\phi\in Dist(G)$. 

\section{Central elements and adjoint invariants}

Let $A$ be an associative superalgebra. The center of $A$, denoted by $Z(A)$, consists of (homogeneous) $a\in A$ such that $ab=(-1)^{|a||b|} ba$ for every (homogeneous) $b\in A$.

Let $A$ be a Hopf superalgebra.
Denote the set 
\[\{y\in A| {\bf ad}(x)y=\epsilon_A(x)y \text{ for all homogeneous } x\in A\}\]
by $A^{{\bf ad}(A)}$. The proof of the following lemma can be modified from Lemma 6.5 of \cite{hab}. See also
Lemma \ref{converseinclusion} below.
\begin{lm}\label{centerviaadjact}
The center $Z(A)$ is contained in $A^{{\bf ad}(A)}$. Moreover, $Z(A)_0 =A_0^{{\bf ad}(A)}$.
\end{lm}
\begin{lm}\label{converseinclusion}
Assume that $A$ is generated by even elements $x$ such that $\Delta_A(x)\in A_0\otimes A_0$ and by additional odd
primitive elements. Then $Z(A)= A^{{\bf ad}(A)}$.
\end{lm}
\begin{proof}
If $y\in A_1^{{\bf ad}(A)}$ and $x$ an even generator of $A$ such that $\Delta_A(x)\in A_0\otimes A_0$, then
$yx=\sum\epsilon(x_1)yx_2=\sum x_1 ys_A(x_2) x_3=\sum x_1 y\epsilon(x_2)=xy$. If $x$ is an odd primitive generator of $A$, then
$yx=\sum (-1)^{|y||x_2|}x_1 ys_A(x_2) x_3=xy-y(-x)+yx=xy+2yx$. Thus $xy=-yx$ and the statement follows.
\end{proof}
\begin{rem}\label{asLieadj}
If $x\in A$ is a primitive element, then ${\bf ad}(x)y=[x, y]$ and $y{\bf ad}(x)=[y, x]$ for every $y\in A$.
\end{rem}
\begin{lm}\label{adjinv}
The set $A^{{\bf ad}(A)}$ is a subsuperalgebra of $A$.
\end{lm}
\begin{proof}
The statement follows from the formula 
\[{\bf ad}(x)(yz)=\sum (-1)^{|x_2||y|} ({\bf ad}(x_1)y)({\bf ad}(x_2)z)\]
for $x, y, z\in A.$
\end{proof}
\begin{lm}\label{adjactionondist}
The conjugation action, desribed above, induces a structure of left $G$-supermodule on $Dist(G)$. It coincides with the structure of the left $Dist(G)$-supermodule 
given by the left adjoint action.
\end{lm}
\begin{proof}
Since any $\mathfrak{m}^{n+1}$ is a subsupercomodule of $(K[G], \mathfrak{con})$, \[Dist_n(G)=(K[G]/\mathfrak{m}^{n+1})^*\] has a natural structure of $G$-supermodule; hence it is a $Dist(G)$-supermodule. By the definition, $(x\cdot y)(\bar f)=(-1)^{|x||y|}y(s_{Dist(G)}(x)\bar f)$, where $\bar f =f+\mathfrak{m}^{n+1}$ and $x, y\in Dist(G)$. Thus
$$(x\cdot y)(\bar f)=\sum (-1)^{|f_1||f_2|+ |x||f_2|+ |x_2||f_1|+|x||y|} x_1(f_1)y(f_2)x_2(s_G(f_3))={\bf ad}(x)y(f).$$
\end{proof}
\begin{cor}\label{normalHopf}
If $H$ is a normal subsupergroup of $G$, then  
$Dist(H)$ is normal in $Dist(G)$. 
\end{cor}
If $G$ is {\it connected} (or {\it pseudoconnected}, in the terminology from \cite{zub1}; see also \S 3 of \cite{zubgrish}), then the converse statement is also true (analogously to Lemma 5.1 of \cite{zubgrish}). 

Combine Lemma \ref{adjactionondist} with Lemma 9.5 from \cite{zub1} to obtain the following statement.
\begin{cor}\label{inv}
If $G$ is connected, then $Dist(G)^{{\bf ad}(Dist(G))}=Dist(G)^G$.
\end{cor}
Using Lemma \ref{adjactionondist} we obtain a homomorphism from $G(C)$ to the group of $C$-linear superalgebra automorphisms of $Dist(G)\otimes C$. Let $i_g$ denote the image of $g\in G(C)$ in $Aut_C(K[G]\otimes C)$. Each operator $i_g$ is locally finite
in the sense that it preserves all finitely generated $A$-subsupermodules $Dist_n(G)\otimes A$.
   
Let $R$ be a normal infinitesimal subsupergroup of $G$. Then $H=Dist(R)=K[R]^*$ is a normal Hopf supersubalgebra of $Dist(G)$.

Recall that for every $n\geq 1$ there is the natural pairing 
\[(Dist_n(G)\otimes C)\times (K[G]/\mathfrak{m}^{n+1}\otimes C)\to C\] 
that allows us to define the pairing
\[(Dist(G)\otimes C)\times (K[G]\otimes C)\to C .\] 
More precisely, if $\phi\in Dist(G)\otimes C$ and $f\in K[G]\otimes C$, then 
for any integer $n\geq 1$ such that $\phi\in Dist_n(G)\otimes C$ we set $\phi(f)=\phi(\overline{f})$, where
$\overline{f}=f+\mathfrak{m}^{n+1}\otimes C$. It is clear that this definition does not depend on $n$, hence it is consistent.

Then $i_g (\phi)(f)=\phi(g^{-1}f)$ for any $g\in G(C), \phi\in Dist(G)\otimes C$ and $f\in K[G]\otimes C$.

The superalgebra $H\otimes C$ is again a right $K[R]\otimes C$-supermodule with respect to the action $\phi r(r')=\phi(r r')$, where $\phi\in H\otimes C$, and $r, r'\in K[R]\otimes C$. Arguing as in Lemma 7.5 of \cite{hab}, we derive that 
$i_g(\phi r)=i_g(\phi)g(r)$.

Let $\nu$ be a right integral on $H$. Since the map $r\mapsto \nu r$ induces an isomorphism of right $K[R]$-supermodules $\Pi^{|\nu|}K[R]\to H$, it can be extended to an isomorphism of $K[R]\otimes C$-supermodules $\Pi^{|\nu|}K[R]\otimes C\to H\otimes C$ by $r\otimes c\mapsto (\nu\otimes 1)(r\otimes c)=
(\nu r)\otimes c$. 

Since the space of integrals on $H$ is one-dimensional, Lemma \ref{integralsareinvariant} implies that there is a character $\chi$ of $G$ such that $i_g(\nu\otimes 1)=\nu\otimes\chi(g)$ for all $g\in G(C)$ and $C\in\mathsf{SAlg}_K$.
For any $G$-supermodule $V$ and for any character $\pi$ of $G$ let $V^{G, \pi}$ denote the subsuperspace $\{v\in V|
\tau_V(v)=v\otimes\pi\}$. For example, if $\pi=1$, then $V^{G, \pi}=V^G$. It is also clear that $v\in V^{G, \pi}$ if and only if
$g(v\otimes 1)=v\otimes\pi(g)$ for every $g\in G(C)$ and $C\in\mathsf{SAlg}_K$.
\begin{lm}\label{equalities}
The isomorphism $f\mapsto \nu f$ induces an isomorphism of superspaces
\[(\Pi^{|\nu|}K[R])^{R, -\chi}\simeq H^{{\bf ad}(H)}=H^R\] and an isomorphism of superspaces \[(\Pi^{|\nu|}K[R])^{G, -\chi}\simeq H^G=Dist(G)^G\bigcap H.\]
\end{lm}
\begin{proof}
Let $\phi\in H$ and $\phi=\nu r$ for $r\in K[R]$. By the definition, 
we have
\[i_g(\phi\otimes 1)=(\nu\otimes 1)((1\otimes\chi(g))g(r\otimes 1))\]
for every $g\in G(C)$ and  $C\in\mathsf{SAlg}_K$. Therefore the maps in the lemma are embeddings of $(\Pi^{|\nu|}K[R])^{R, -\chi}$ and $(\Pi^{|\nu|}K[R])^{G, -\chi}$ into $H^R$ and $H^G$, respectively. 

Conversely, assume that $\phi\in H^G$ (the case $\phi\in H^R$ is analogous). Then $\phi=\nu r$ for some $r\in R$. Set $C=K[G]$ and $g=id_{K[G]}\in G(K[G])$. If $\mathfrak{con}(r)=\sum r_1\otimes f_2$, then
\[i_g(\phi\otimes 1)=\sum \nu r_1\otimes \chi f_2 =\nu r\otimes 1.\]
Therefore $\mathfrak{con}(r)=r\otimes \chi^{-1}$, which means $r \in R^{G, -\chi}$.
\end{proof}
\begin{rem}\label{isomuptogrouplikeelement}
The proof of Lemma \ref{equalities} also shows that the $G$-supermodules $H$ and $\Pi^{|\nu|}K[R]\otimes\chi$ are isomorphic. The isomorphism is defined by $r\otimes\chi\mapsto \nu r$ 
for $r\in K[R]$.
\end{rem}
From now on $R=G_r$ is an $r$-th Frobenius kernel of $G$ (cf. \cite{jan, zub1}). 
Let us choose homogeneous elements $f_i\in K[G]$ such that $f_i +\mathfrak{m}^2$ form a homogeneous basis of $\mathfrak{m}/\mathfrak{m}^2=Lie(G)^*$. Assume that $|f_i|=0$ whenever $1\leq i\leq n_0\leq n$, and $|f_i|=1$ whenever $n_0<i\leq n$, and denote $n-n_0$ by $n_1$.
\begin{rem}\label{basesforFrob}
Suppose that $G_{ev}$ is reduced.
Combining Theorem 4.5 of \cite{amas}, with I.9.6 of \cite{jan}, we see that $K[G_r]$ has a basis consisting of all monomials $\prod_{1\leq i\leq t}\bar{f}_i^{a_i}$, where 
$\bar{f}_i$ denotes the image of $f_i$ in $K[G_r]$, $0\leq a_i\leq p^r -1$ provided $|f_i|=0$, and $a_i=0, 1$ provided $|f_i|=1$. 
\end{rem}
Denote by $\nu=\nu_r$ a right integral on $Dist(G_r)$ and by $\chi=\chi_r$ the character of $G$ corresponding to right integrals.
Denote by $Ad$ the adjoint representation of $G$ on $Lie(G)=Dist_1(G)^+$, that is $Ad(g)=i_g|_{Lie(G)\otimes C}$ for $g\in G(C)$.
\begin{pr}\label{berinfinitesimal}
We have the following statements  
\begin{enumerate}
\item $|\nu_r|=n_1 \pmod 2$,
\item the character $\chi_r$ maps 
$g$ to $Ber(Ad(g))^{p^r-1}\det(Ad(g)_{11})^{p^r}$. 
\end{enumerate}
\end{pr}
\begin{proof}
Set $V=\sum_{1\leq i\leq n} K\bar{f}_i$.
An operator $\Phi\in GL(V)(C)$ for $C\in \mathsf{SAlg}_K$
has the following matrix in the $n_0|n_1$-block form (with respect to the above basis)  
$$\Phi=\left(\begin{array}{cc}
\Phi_{00} & \Phi_{01} \\
\Phi_{10} & \Phi_{11}
\end{array}
\right).
$$
Since $\Phi_{00}$ and $\Phi_{11}$ are invertible, $\Phi$ has the decomposition 
$$
\Phi=\left(\begin{array}{cc}
I_{n_0} & \Phi_{01}\Phi_{11}^{-1} \\
0 & I_{n_1}
\end{array}\right)\left(\begin{array}{cc}
\Phi_{00}-\Phi_{01}\Phi_{11}^{-1}\Phi_{10} & 0 \\
0 & \Phi_{11}
\end{array}\right)\left(\begin{array}{cc}
I_{n_0} & 0 \\
\Phi_{11}^{-1}\Phi_{10} & I_{n_1}
\end{array}
\right)
$$
Set $F=(\prod_{1\leq i\leq n_0}\bar{f}_i)^{p^r -1}\prod_{n_0 < i\leq n}\bar{f}_i$. The first and the third factors of the above decomposition of $\Phi$ act on $F$ trivially.
Arguing as in Proposition I.9.7 of \cite{jan}, we obtain that $\Phi$ maps the element $F\otimes 1$ 
to $F\otimes c(\Phi)$, where 
\[c(\Phi)=\det(\Phi_{00}-\Phi_{01}\Phi_{11}^{-1}\Phi_{10})^{p^r -1}\det({\Phi_{11}})=Ber(\Phi)^{p^r -1}\det(\Phi_{11})^{p^r}.\] 
Since $Ber(\Phi)$ and $\det(\Phi_{11})^{p^r}$ are group homomorphisms from $GL(V)(C)$ to $C_0^{\times}$,
$c(Ad(g))=Ber(Ad(g))^{-(p^r -1)}\det(Ad(g)_{11})^{-p^r}.$ Observe also that $K[G_r]F=K F$.
Using Remark \ref{isomuptogrouplikeelement} we complete our proof as in Proposition I.9.7 of \cite{jan}.
\end{proof}

\section{Central elements in $Dist(GL(m|n))$}

Throughout this section $G=GL(m|n)$.
\begin{lm}\label{integralascentral}
In the notation of Proposition \ref{berinfinitesimal}, $\chi_r=1$ for every $r\geq 1$.
\end{lm}
\begin{proof}
Set $t_{ij}=c_{ij}-\delta_{ij}$ for $1\leq i, j\leq m+n$. $V=
\sum_{1\leq i, j\leq m+n} Kt_{ij}$ is a supersubmodule of $K[G]$ with respect to the adjoint action.
Besides, $t_{ij}+\mathfrak{m}^2$ form a basis of $Lie(G)^*$. By Lemma 13.5 of \cite{zub3}, it is enough to check that
$\chi|_T =1$. It can be easily seen that
the matrix $Ad(t)$ is also diagonal. Moreover, the diagonal entries of $Ad(t)$ are $t_i t_j^{-1}$ for $1\leq i, j\leq m+n$, where 
$$t=\left(\begin{array}{cccc}
t_1 & 0 & \dots & 0 \\
0 & t_2 & \dots & 0 \\
\vdots & \vdots & \vdots & \vdots\\
0 & 0 & \dots & t_{m+n}
\end{array}\right)\in T(C).$$
Thus $Ber(Ad(t))=\det(Ad(t)_{22})=1$. 
\end{proof}
Observe that $|\nu_r|= 2mn =0\pmod 2$, that is each $\nu_r$ is even.
\begin{cor}\label{corko}
The map $f\to \nu_r f$ induces isomorphisms $K[G_r]^{G_r}\simeq Dist(G_r)^{G_r}$ and $K[G_r]^G\simeq Dist(G_r)\bigcap Dist(G)^G$.
In particular, each $\nu_r$ belongs to $Z(Dist(G))$ and hence $\nu_r$ is a two-sided integral.
\end{cor}
Let $e_{ij}$ be the generators of $gl(m|n)=Lie(GL(m|n))$ that are dual to $t_{ij}+\mathfrak{m}^2$. In other words, 
$e_{ij}(t_{kl})=\delta_{i k}\delta_{j l}$ for $1\leq i, j, k,l\leq m+n$. The Hopf superalgebra $Dist(G)$ can be identified with
$K\otimes_{\mathbb{Z}} U_{\mathbb{Z}}(gl(m|n))$, where $U_{\mathbb{Z}}(gl(m|n))$ is a $\mathbb{Z}$-supersubalgebra of
$U_{\mathbb{Q}}(gl(m|n))$ generated by the elements
\[e_{ij}^{(t)}=\frac{e_{ij}^t}{t!} \mbox{ and } 
\left(\begin{array}{c}
e_{ii} \\
s
\end{array}\right)=\frac{e_{ii}(e_{ii}-1)\ldots (e_{ii}-s+1)}{s!},\]
where $i\neq j$ and $s, t\geq 0.$
Additionally, if $|e_{ij}|=1$, then $t=0, 1$. 

The Hopf superring $U_{\mathbb{Z}}(gl(m|n))$ has a $\mathbb{Z}$-basis (as a free $\mathbb{Z}$-supermodule)
consisting of ordered products of the above generators, say (see Lemma 3.1 of \cite{brunkuj})
\[\prod_{|e_{ij}|=1}e_{ij}^{a_{ij}}\prod_{1\leq i\leq m+n}\left(\begin{array}{c}
e_{ii} \\
s_i
\end{array}\right)\prod_{|e_{ij}|=0}e_{ij}^{(d_{ij})}.\]
One can choose a different ordering, for example 
\[\prod_{|e_{ij}|=1}e_{ij}^{a_{ij}}\prod_{|e_{ij}|=0, i>j}e_{ij}^{(d_{ij})}\prod_{1\leq i\leq m+n}\left(\begin{array}{c}
e_{ii} \\
s_i
\end{array}\right)\prod_{|e_{ij}|=0, i<j}e_{ij}^{(d_{ij})}, \]
and many additional ones. 

Since \[\Delta_{Dist(G)}(Dist(G_{ev}))\subseteq Dist(G_{ev})\otimes Dist(G_{ev})\] and 
\[\Delta_{Dist(G_r)}(Dist(G_{ev, r}))\subseteq Dist(G_{ev, r})\otimes Dist(G_{ev, r})\] as well, Lemma \ref{converseinclusion} implies the following statement.
\begin{lm}\label{apllyingtogenlin}
The center of $Dist(G)$ coincides with $Dist(G)^{{\bf ad}(Dist(G))}$ and the center of $Dist(G_r)$ coincides with $Dist(G_r)^{{\bf ad}(Dist(G_r))}$. 
\end{lm}
\begin{lm}\label{basis}
The elements 
\[\prod_{|t_{ij}|=1}t_{ij}^{k_{ij}}\prod_{1\leq i\leq m+n}
t_{ii}^{l_i}
\prod_{|t_{ij}|=0}t_{ij}^{r_{ij}}\]
such that $k_{ij}\in\{0, 1\}$ and $0\leq l_i, r_{ij}\leq p^r -1$ form a basis of $K[G_r]$.

The elements
\[\prod_{|e_{ij}|=1}e_{ij}^{a_{ij}}\prod_{1\leq i\leq m+n}\left(\begin{array}{c}
e_{ii} \\
s_i
\end{array}\right)\prod_{|e_{ij}|=0}e_{ij}^{(d_{ij})}\]
such that $a_{ij}\in\{0, 1\}$ and $0\leq s_i, d_{ij}\leq p^r -1,$ form a basis of $Dist(G_r)$.

The second statement remains valid if we arbitrarily change the order of the appearing factors.
\end{lm}
\begin{proof}
Since $G_{ev}$ is reduced, the first statement follows from Remark \ref{basesforFrob}. 
Additionally, $Dist(G_r)=K[G_r]^*$ and $Dist(G_r)$ is generated (as a vector superspace) by the above products. 
Comparison of dimensions implies the second statement.
The final part follows from Lemma 3.1 and Theorem 3.2 of \cite{brunkuj}.
\end{proof}
\begin{rem}\label{otherbasis}
The elements \[\prod_{|c_{ij}|=1}c_{ij}^{k_{ij}}\prod_{1\leq i\leq m+n}
c_{ii}^{l_i}
\prod_{|c_{ij}|=0}c_{ij}^{r_{ij}}\]
also form a basis of the 
superalgebra $K[G_r]$, subject to the same restrictions on $k_{ij}, l_i, r_{ij}$.
\end{rem}
In what follows the products
$$\prod_{1\leq i\leq m+n}\left(\begin{array}{c}
e_{ii}-a_i \\
s_i
\end{array}\right) \ \mbox{and} \prod_{1\leq i\leq m+n}c_{ii}^{l_i},$$ 
where 
$a=(a_1,\ldots, a_{m+n})$ is a vector with integer coordinates and 
$s=(s_1,\ldots , s_{m+n})$ and $l=(l_1, \ldots , l_{m+n})$ are vectors with non-negative integer coordinates,
will be denoted by $\left(\begin{array}{c}
e-a \\
s
\end{array}\right)$ and $c^{l}$ respectively. Let $|s|$ denote the sum $\sum_{1\leq i\leq m+n} s_i$.

Following page 64 of \cite{hab}, denote 
\[\Delta^{(r)}_{T, a}=\sum_{s, \ 0\leq s_1, \ldots, s_{m+n}\leq p^r -1}(-1)^{|s|}
\left(\begin{array}{c}
e -a\\
s
\end{array}\right).\]

\begin{lm}\label{overeven}
For every $1\leq k, l, s\leq m+n$ and $t\geq 1$ we have
\[[e_{kl}^{(t)}, e_{ss}]=t(\delta_{ls}-\delta_{ks})e_{kl}^{(t)}.\]
\end{lm}
\begin{proof}
Working in the superring $U_{\mathbb{Z}}(gl(m|n))$ we see that
$$[e_{kl}^{(t)}, e_{ss}]=\frac{1}{t!}\sum_{0\leq s\leq t-1}e_{kl}^s [e_{kl}, e_{ss}]e_{kl}^{t-s-1}.$$
Since $[e_{kl}, e_{ss}]=(\delta_{ls}-\delta_{ks})e_{kl}$, our statement follows.
\end{proof}
\begin{lm}\label{oddovertorus}
For every $i, j, k$ and $t\geq 1$ we have
\[\left(\begin{array}{c}
e_{kk} \\
s
\end{array}\right)e_{ij}^{(t)}=e_{ij}^{(t)}\left(\begin{array}{c}
e_{kk}+t(\delta_{ki}-\delta_{kj}) \\
s
\end{array}\right).\]
\end{lm}
\begin{proof}
Since 
$(e_{kk}-a)e_{ij}=e_{ij}(e_{kk}+\delta_{ki}-\delta_{kj}-a)$ for any $a\in \mathbb{Z}$, the statement follows.
\end{proof}
The supergroup $U^+$ has a series of normal subsupergroups
\[1\leq U^+(m+n-1)\leq\ldots\leq U^+(1)=U^+\]
such that $K[U^+(d)]=K[U^+]/I^+_d$, where an Hopf superideal $I^+_d$ is generated by the elements $c_{ij}$ with $1\leq j-i < d$.  
Symmetrically, $U^-$ has a series of normal subsupergroups
$$1\leq U^-(m+n-1)\leq\ldots\leq U^-(1)=U^-$$
such that $K[U^-(d)]=K[U^-]/I^-_d$, where an Hopf superideal $I^-_d$ is generated by the elements $c_{ij}$ with $1\leq i-j < d$. Additionally, each Frobenius kernel $U^{\pm}_r$ has a similar series of normal subsupergroups $U^{\pm}_r(d)=U^{\pm}(d)\bigcap G_r$. 

Let $R$ be an algebraic supergroup such that $R=U\rtimes M$ is a semi-direct product of its normal finite subsupergroup $U$ and purely even subsupergroup $M$. Let $\nu$ be a right integral on $Dist(U)$ and $\mu$ be a right integral on $Dist(M)$.
Since $Dist(U)$ is normal in $Dist(R)$, we can use Lemma \ref{integralsareinvariant} to derive that
${\bf ad}(x)\nu=x(\rho)\nu,$ for $x\in Dist(R)$ and an appropriate character $\rho\in X(R)$.
\begin{lm}\label{semidirect}
If any character of $R$ is trivial on $U$, then $(\mu\nu)\rho^{-1}=(\mu\rho^{-1})\nu=\nu\mu$ is an right integral on $Dist(R)$. 
\end{lm}
\begin{proof}
Observe that $Dist(R)=Dist(U)Dist(M)$ and since $\rho$ is trivial on $U$, we have $\rho\in K[M]\subseteq K[R]$. Since $\rho$ is a group-like element, 
$(xy)\rho^k=(x\rho^k) (y\rho^k)$ for every $x, y\in Dist(R)$ and $k\in\mathbb{Z}$. Additionally, $y\rho^k=y$ for every $y\in Dist(U)$. In fact, for every $f\in K[R]$ we have 
\[y\rho^k(f)=y(\rho^k f)=\sum y_1(\rho^k)y_2(f)=\sum y_1(1)y_2(f)=y(f).\]
From the definition of comultiplication in the algebra of distributions it follows that 
\[x=\sum x_1(\rho) (x_2\rho^{-1})\]
and 
\[\sum s_M(x_1)(x_2\rho^{-1})=x(\rho^{-1})\] for every $x\in Dist(M)$. Thus
\[\nu x=\sum x_1(\rho)\nu (x_2\rho^{-1})=\sum x_1\nu s_M(x_2)(x_3\rho^{-1})=(\sum x_1 x_2(\rho^{-1}))\nu=(x\rho^{-1})\nu .\]
Finally, we have 
\[(\mu\nu)\rho^{-1}(yx)=(\mu\nu (y\rho) (x\rho))\rho^{-1}=x(1)y(\rho)(\mu\nu)\rho^{-1}=
(xy)(1)(\mu\nu)\rho^{-1}.\]
\end{proof}
\begin{pr}\label{someintegralforunip}
Any element
$$\prod_{|e_{ij}|=1, j-i\geq d}e_{ij} \prod_{|e_{ij}|=0, j-i\geq d}e_{ij}^{(p^r-1)}=\prod_{|e_{ij}|=0, j-i\geq d}e_{ij}^{(p^r-1)}\prod_{|e_{ij}|=1, j-i\geq d}e_{ij},$$ 
where the products $\prod_{|e_{ij}|=1, j-i\geq d}e_{ij}$ and $\prod_{|e_{ij}|=0, j-i\geq d}e_{ij}^{(p^r-1)}$ can be taken in any order, is a two-sided integral on $Dist(U_r^+(d))$. Moreover, it is a central element in $Dist(U^+)$.
\end{pr}
\begin{proof}
Let $V^+$ be the kernel of the natural epimorphism $U^+\to U^+_{ev}$. Then $U^+_r(d)=V^+(d)\rtimes U^+_{ev, r}(d)$, where $V^+(d)=V^+\bigcap U^+(d)$. 
In particular, $V^+(d)\unlhd U^+$ for every $d\geq 1$.
It is clear that $Dist(V^+(d))$ is generated by all odd primitive elements $e_{ij}$ with $j-i\geq d$. Since $V^+(d)$ is abelian, 
$Dist(V^+(d))$ is supercommutative. Therefore $\prod_{|e_{ij}|=1, j-i\geq d}e_{ij}$ is a two-sided integral on $Dist(V^+(d))$.
Since unipotent supergroups have only trivial characters, Lemma \ref{integralsareinvariant} implies that 
$\prod_{|e_{ij}|=1, j-i\geq d}e_{ij}$ is central in $Dist(U^+)$. Finally, 
Lemma \ref{semidirect} combined with Proposition 6.7 of \cite{hab} concludes the proof.
\end{proof}
\begin{pr}\label{symmversion}
Symmetrically, any element
$$\prod_{|e_{ij}|=1, i-j\geq d}e_{ij} \prod_{|e_{ij}|=0, i-j\geq d}e_{ij}^{(p^r-1)}=\prod_{|e_{ij}|=0, i-j\geq d}e_{ij}^{(p^r-1)}\prod_{|e_{ij}|=1, i-j\geq d}e_{ij},$$ 
where the products $\prod_{|e_{ij}|=1, i-j\geq d}e_{ij}$ and $\prod_{|e_{ij}|=0, i-j\geq d}e_{ij}^{(p^r-1)}$ can be taken in any order, is a two-sided integral on $Dist(U_r^-(d))$. Moreover, it is a central element in $Dist(U^-)$.
\end{pr}
Let $u^+_0$ and $u^-_0$ denote $\prod_{|e_{ij}|=0, j > i}e_{ij}^{(p^r-1)}$ and $\prod_{|e_{ij}|=0, i > j}e_{ij}^{(p^r-1)}$
correspondingly. Besides, denote $\prod_{|e_{ij}|=1, j> i}e_{ij}$ and $\prod_{|e_{ij}|=1, i> j}e_{ij}$ by $u^+_1$ and $u^-_1$
correspondingly. 

Let $P^-$ be a parabolic subsupergroup of $G$ consisting of all matrices that have the right upper block of size $m\times n$ equal to zero matrix. Then $P^-\simeq V^-\rtimes G_{ev}$, where $V^-$ is the kernel of the natural epimorphism $U^-\to U^-_{ev}$. Moreover, $P_r^-\simeq V^-\rtimes G_{ev, r}$ for every $r\geq 1$.
Observe that $G_{ev}$ acts on $u^-_1$ via the character $\rho=\det(C_{11})^m\det(C_{00})^{-n}$ and on $u^+_1$ via the character $\rho^{-1}$. 

It follows from Corollary 6.10 of \cite{hab} that $\Delta_T^{(r)} u^-_0 u^+_0$ and $\Delta_T^{(r)} u^+_0 u^-_0$ are two-sided integrals on $Dist(G_{ev})$. 
We will show that $\Delta_T^{(r)} u^-_0 u^+_0=u^-_0 u^+_0\Delta_T^{(r)}$ and $\Delta_T^{(r)} u^+_0 u^-_0=u^+_0 u^-_0\Delta_T^{(r)}$.
In fact, if $Dist(T)$ acts on $u^-_0$ via a character $\chi$, then $Dist(T)$ acts on $u^+_0$ via $\chi^{-1}$.
By Lemma \ref{semidirect} we have
$$\Delta_T^{(r)} u^-_0 u^+_0=u^-_0(\Delta_T^{(r)}\chi)u^+_0=u^-_0 u^+_0(\Delta_T^{(r)}\chi)\chi^{-1}=
u^-_0 u^+_0\Delta_T^{(r)}$$
and analogously for the second equation. 

Symmetrically, one can consider the opposite parabolic subsupergroup $P^+$ of $G$ consisting of all matrices that have the left lower block of size $n\times m$ equal to zero matrix. 
Again, $P^+\simeq V^+\rtimes G_{ev}$ and $P^+_r\simeq V^+\rtimes G_{ev, r}$ for every $r\geq 1$.

\begin{pr}\label{forP}
The elements 
\[(u^-_0 u^+_0\Delta_T^{(r)})\rho^{-1} u^-_1=u^-_1 u^-_0 u^+_0\Delta_T^{(r)}\] and 
\[(u^+_0 u^-_0\Delta_T^{(r)})\rho^{-1} u^-_1=u^-_1 u^+_0 u^-_0\Delta_T^{(r)}\]
are right integrals on $Dist(P_r^-)$, and the elements 
\[(u^-_0 u^+_0\Delta_T^{(r)})\rho^{-1} u^+_1=u^+_1 u^-_0 u^+_0\Delta_T^{(r)}\] and 
\[(u^+_0 u^-_0\Delta_T^{(r)})\rho^{-1} u^+_1=u^+_1 u^+_0 u^-_0\Delta_T^{(r)}\]
are right integrals on $Dist(P^+_r)$. 
\end{pr}

Set $W_i=\sum_{m+n-i+1\leq s\leq m+n} Kw_s$ for $1\leq i\leq m+n, S=Kw_m+Kw_{m+1}$ and $L=\sum_{i\neq m, m+1}Kw_i$. The stabilizer of the flag
$$W_1\subseteq\ldots\subseteq W_{n-1}\subseteq W_{n+1}\subseteq\ldots\subseteq W_{m+n}=W$$
is denoted by $Q$. It is easy to see that $Q\simeq U\rtimes H$, where $U$ is the largest subsupergroup of $U^-$ whose elements act trivially on $W_{n+1}/W_{n-1}$, and 
$H\simeq R\times T'$, where 
\[T'(C)=\{t\in T(C); t|_{S\otimes C}=id_{S\otimes C}\}\]
and  
\[R(C)=\{g\in G(C); g|_{L\otimes C}=id_{L\otimes C}\}\]
for $C\in\mathsf{SAlg}_K$.
It is clear that $R\simeq GL(1|1)$ and $U^-\leq Q$. Finally, for every $r\geq 1$ we have $Q_r\simeq U_r\rtimes H_r$.

Arguing as in Proposition \ref{someintegralforunip} we see that the element
$$x=\prod_{|e_{ij}|=1, i> j, e_{ij}\neq e_{m+1, m}}e_{ij}\prod_{|e_{ij}|=0, i> j}e_{ij}^{(p^r-1)}=$$$$\prod_{|e_{ij}|=0, i> j}e_{ij}^{(p^r-1)}\prod_{|e_{ij}|=1, i> j, e_{ij}\neq e_{m+1, m}}e_{ij}$$
is a two-sided integral on $Dist(U_r)$. Since $U_r\unlhd Q$, $x$ (super)commutes with $Dist(U^-)$ and $e_{m, m+1}$. 
Indeed, there is a character $\chi\in X(Q)=X(H)$ such that ${\bf ad}(e_{m,m+1})x$ $=e_{m, m+1}(\chi)x$. Since
$e_{m, m+1}(\chi)=e_{m, m+1}(1)=0$, we have ${\bf ad}(e_{m,m+1})x=e_{m, m+1}x-(-1)^{|x|}xe_{m, m+1}=0$.
\begin{rem}\label{aboutprimitives}
The last statement can be generalized as follows. If $x$ is as above, $\chi$ is the character corresponding to $x$ and $y\in Lie(G)$, then $yx=x((-1)^{|x||y|}y+y(\chi))$.  
\end{rem}
\begin{tr}\label{whatintegralis}
The central element $\nu_r$ equals
$u^+_1 u^+_0 u^-_1 u^-_0\Delta_T^{(r)} =\Delta_T^{(r)} u^+_1 u^+_0 u^-_1 u^-_0$.
\end{tr}
\begin{proof}
Denote the above product by $y$. The same arguments as in Proposition \ref{forP} imply the equation $u^+_1 u^+_0 u^-_1 u^-_0\Delta_T^{(r)} =\Delta_T^{(r)} u^+_1 u^+_0 u^-_1 u^-_0$. 
Moreover, $u^{\epsilon}_0$ commutes with $u^{\mu}_1$ for each 
$\epsilon, \mu\in\{\ -, +\}$. For example, $u^+_0 u^-_1 = u^-_1 (u^+_0\rho)=u^-_1 u^+_0$.

Thus $y=u^+_1 u^-_1 u^+_0 u^-_0\Delta_T^{(r)}$ and 
Proposition \ref{forP} implies that $yx=0$ for every $x\in Dist(P_r^- )^+$.
If $j > m+1$ or $i< m$, then $e_{ij}=[e_{i, m+1}, e_{m+1, j}]$ and $e_{ij}=[e_{i m}, e_{m j}]$ respectively. Using the arguments on page 68 of \cite{hab} it remains to consider only the case $i=m, j=m+1$ and we need to show that $ye_{m,m+1}=0$. 
Using Lemma \ref{oddovertorus} we infer that 
$\Delta_T^{(r)} e_{m, m+1}=e_{m, m+1}\Delta^{(r)}_{T, a}$, where $a_i=0$ for $i\neq m, m+1$, and $a_m=-1, a_{m+1}=1$. 
Thus $$ye_{m, m+1}=u^+_0 u^+_1 e_{m+1, m} x e_{m, m+1}\Delta^{(r)}_{T, a}=
(-1)^{|x|}u^+_0 u^+_1 e_{m+1, m}e_{m,m+1} x\Delta^{(r)}_{T, a}.$$
Using the equation $$e_{m+1, m}e_{m,m+1}=e_{m+1, m+1}+e_{mm}-e_{m, m+1}e_{m+1, m}$$ and Proposition \ref{someintegralforunip}, one can continue as
$$ye_{m, m+1}=(-1)^{|x|}u^+_0 u^+_1 (e_{m+1, m+1}+e_{mm}) x \Delta^{(r)}_{T, a}.$$
On the other hand, Lemma \ref{overeven} implies that $[x, e_{mm}]=(n-1 +(p^{r}-1)n)x=-x$ and $[x, e_{m+1, m+1}]=
-(m-1+(p^r -1)m)x=x$,
and consequently $[x, e_{m+1, m+1}+e_{mm}]=0$. Therefore, 
$ye_{m, m+1}=(-1)^{|x|}u^+_0 u^+_1 x (e_{m+1, m+1}+e_{mm}) \Delta^{(r)}_{T, a}$ and by
Corollary 6.4 of \cite{hab}, we derive $(e_{m+1, m+1}+e_{m m})\Delta^{(r)}_{T, a}=\Delta^{(r)}_{T, a}-\Delta^{(r)}_{T, a}=0$.  
\end{proof}

\section{Harish-Chandra homomorphism}

As in the previous section, $G=GL(m|n)$. Let us recall some standard properties of $G$-supermodules and $G_r$-supermodules.
First of all, Lemma \ref{equivalence} shows that the category of $G_r$-supermodules is naturally equivalent to the category of $Dist(G_r)$-supermodules. A simple $G_r$-supermodule is isomorphic to the $G_r$-head $L_r(\lambda)$ of $Dist(G_r)\otimes_{Dist(B_r^+)} K_{\lambda}$
or to its parity shift $\Pi L_r(\lambda)$ (cf. \cite{kuj}, Theorem 5.4). Moreover, $L_r(\lambda)\simeq L_r(\mu)$ if and only if
$\lambda-\mu\in p^r X(T)$. Observe also that $L_r(\lambda)\neq 0$ for arbitrary weight $\lambda\in X(T)$.

The category of $G$-supermodules is equivalent to the category of {\it integrable} \linebreak $Dist(G)$-supermodules (cf. \cite{brunkuj}, Corollary 3.5). Since $Dist(G)=Dist(G_{ev})Dist(G_r)$, one can mimic the proofs of Proposition 4.3 and Proposition 4.4 from
\cite{shuweiq} to show that
the simple $G$-supermodule $L(\lambda)$ of highest weight $\lambda\in X(T)^+$ remains simple as $G_r$-supermodule whenever
$\lambda\in X_r(T)^+$, where 
\[X_r(T)^+=\{\lambda\in X(T) | 0\leq \lambda_i-\lambda_{i+1} < p^r \mbox{ for } i\neq m,  1\leq i\leq m+n-1\}.\]   
It is easy to see that for any weight $\lambda\in X(T)$ there are $r\geq 1$ and $\mu\in X_r(T)^+$ such that $\lambda-\mu\in p^r X(T)$. In particular, $L_r(\lambda)\simeq L(\mu)|_{G_r}$.

In addition to the Bruhat-Tits order on $X(T)$ let us define a partial order $\preceq$ on $X(T)$ by $\mu\preceq\lambda$ if and only if 
$\mu_i\leq\lambda_i$ for each $1\leq i\leq m+n$. 
Denote by $X^{(r)}(T)$ the set $$\{\lambda\in X(T)| 0\preceq\lambda\preceq (p^r -1, \ldots, p^r-1)\}.$$ 

Denote $Z(Dist(G))$ by $Z$ and $Dist(G_r)^G=Z\bigcap Dist(G_r)$ by $Z_r$.
Then $Z=\bigcup_{r\geq 1} Z_r$ 
is a locally finite (super)algebra. The (Jacobson) radical of $Z$ is equal to its upper nil radical (cf. \cite{jb}, I, \S 10, Theorem 1).
\begin{lm}\label{adecomposition}
The superspace $Dist(G)$ can be decomposed as $$Dist(G)=Dist(T)\oplus (Dist(G)Dist(U^+)^+ +Dist(U^-)^+ Dist(G)).$$
\end{lm}
\begin{proof}
Denote $Dist(G)Dist(U^+)^+ +Dist(U^-)^+ Dist(G)$ by $J$.
$Dist(G)$ is isomorphic to $Dist(U^-)\otimes Dist(T)\otimes Dist(U^+)$ as a $Dist(U^-)\times
Dist(U^+)$-superbimodule. Thus $Dist(G)Dist(U^+)^+$ has a basis consisting of the products
$$\prod_{|e_{ij}|=1, i> j}e_{ij}^{d_{ij}}\prod_{|e_{ij}|=0, i> j}e_{ij}^{(t_{ij})}\left(\begin{array}{c}
e \\
s
\end{array}\right)\prod_{|e_{ij}|=1, i< j}e_{ij}^{d'_{ij}}\prod_{|e_{ij}|=0, i< j}e_{ij}^{(t'_{ij})}$$
with $\sum d'_{ij} +\sum t'_{ij}>0$. Symmetrically, $Dist(U^-)^+ Dist(G)$ has a basis consisting of the products
$$\prod_{|e_{ij}|=1, i> j}e_{ij}^{d_{ij}}\prod_{|e_{ij}|=0, i> j}e_{ij}^{(t_{ij})}\left(\begin{array}{c}
e \\
s
\end{array}\right)\prod_{|e_{ij}|=1, i< j}e_{ij}^{d'_{ij}}\prod_{|e_{ij}|=0, i< j}e_{ij}^{(t'_{ij})}$$
with $\sum d_{ij} +\sum t_{ij}>0$. Hence $J$ has a basis consisting of the above products
such that $\sum d'_{ij} +\sum t'_{ij}+ \sum d_{ij} +\sum t_{ij}>0$. In particular, $Dist(T)\bigcap
J=0$. It remains to observe that $Dist(G)=Dist(T)+J$.
\end{proof}
Every $x\in Dist(G)$ can be uniquely presented as $x=x_1+x_2$, where $x_1\in Dist(T)$ and $x_2\in J$. Denote $x_1$ by $h(x)$.
Notice that $|x|=1$ implies $h(x)=0$. In other words, the linear map $h : Dist(G)\to Dist(T)$ is a superspace homomorphism.

Since $L(\lambda)$ is irreducible as a $K[G]$-comodule, that is $L(\lambda)$ is an absolutely irreducible $G$-supermodule (cf. \cite{zub2}), 
any central element $z$ acts on  $L(\lambda)$ as
$zv=z_{\lambda}v$ for each $v\in L(\lambda)$, where $z_{\lambda}\in K$.
\begin{lm}\label{actaszero} 
If a central element $z$ satisfies the condition $z_{\lambda}=0$ for every $\lambda\in X(T)^+$, then $z$ is nilpotent.
\end{lm}
\begin{proof}
Since any simple $G_r$-supermodule is the restriction to $G_r$ of a simple $G$-supermodule, one can modify the proof of Lemma 7.3 of \cite{hab}. 
\end{proof}
The following Proposition generalizes results from \cite{har, howe}.

\begin{pr}\label{harish-chandra} 
The linear map $h$ induces a superalgebra homomorphism $Z\to Dist(T)$. Moreover, the kernel of $h|_Z$ is contained in the radical of $Z$.
\end{pr}
\begin{proof}
If $x$ is a homogeneous central element, then for any $y\in Dist(G)$ we have $y_2 x\in J$. Thus $h(yx)=h(y_1 x)=h(y_1 x_1 + y_1 x_2)$. Since $Dist(U^-)$ is a normal Hopf supersubalgebra of $Dist(B^-)$, $uz=\sum u_1z s_{Dist(T)}(u_2) u_3\in Dist(U^-)Dist(T)$ for any $u\in Dist(T), z\in Dist(U^- )$. Thus $y_1x_2\in J$ and $h(yx)=h(y)h(x)$.

Let $z\in Z\bigcap J$. Then $z=z' +z"$, where $z'\in Dist(G)Dist(U^+)^+, z"\in Dist(U^-)^+ Dist(G)$. Let $v_{\lambda}\in L(\lambda)_{\lambda}\setminus 0$. Since $Dist(U^+)^+ v_{\lambda}=0$, it follows that
$zL(\lambda)\subseteq Dist(U^-)^+ L(\lambda)$. On the other hand, $W=Dist(U^-)^+L(\lambda)$ is a $T$-supermodule whose all non-zero weight components $W_{\mu}$ satisfy $\mu <\lambda$, that is $W$ is a proper supersubspace of $L(\lambda)$. Thus $z_{\lambda}=0$ for every $\lambda$. Lemma \ref{actaszero} concludes the proof.
\end{proof}
\begin{rem}\label{actionofHC}
The arguments in the above proof of Proposition \ref{harish-chandra} imply that 
$h(z)v=z_{\lambda}v$ for every $z\in Z(Dist(G))$, $\lambda\in X(T)^+$ and $v\in L(\lambda)_{\lambda}$.
\end{rem}
\begin{rem}\label{restrictiononFrob}
For any $r\geq 1$ one can define $J_r=J\bigcap Dist(G_r)$. It is easy to see that 
$J_r=Dist(G_r)Dist(U_r^+)^+ +Dist(U_r^-)^+ Dist(G_r)$ and $Dist(G_r)=Dist(T_r)\oplus J_r$. Again, one can define
a superspace homomorphism $h_r : Dist(G_r)\to Dist(T_r)$ that coincides with $h|_{Dist(G_r)}$ and induces a superalgebra homomorphism $Z_r\to Dist(T_r)$. 
\end{rem}
Let $G$ be an algebraic supergroup that acts on a (not necessary supercommutative) superalgebra $A$ by superalgebra automorphisms. In other words, $A$ is a $G$-supermodule and for any superalgebra $C\in\mathsf{Salg}_K$ the induced 
homomorphism from $G(C)$ to the group of $C$-linear locally finite automorphisms of $A\otimes C$ preserves the superalgebra structure of $A\otimes C$. It is easy to see that the latter condition is equivalent to $\tau_A$ being a superalgebra morphism.
\begin{rem}\label{actionofoddunipotent}
Let $G=G_a^-$ act on a superalgebra $A$. Then $\tau_A(a)=a\otimes 1 +\delta(a)\otimes x$, where $\delta\in End_K(A)$ is a locally finite odd endomorphism of the superspace $A$ such that $\delta^2=0$. Then $\tau_A$ is a superalgebra morphism if and only if $\delta$ is a right superderivation of $A$.
\end{rem}
Assume again that $G=GL(m|n)$. Consider the element $g\in U_{ij}(K[x])$ such that $g=E +xE_{ij}$, where $i\neq j, |x|=|e_{ij}|=1$. Here $E$ is the identity matrix of size $(m+n)\times (m+n)$
and $E_{ij}$ is the matrix of size $(m+n)\times(m+n)$, whose entry at the $(i,j)$-th position equals 1 and all remaining entries are zeroes. 
The conjugation action of $U_{ij}(K[x])$ on $K[G]\otimes K[x]$ is given as follows.
$$g(c_{kl})=
\left\{\begin{array}{ll}
c_{kl}, &\mbox{ if } k\neq i \ \mbox{and} \ l\neq j , \\
c_{il} -xc_{jl}, & \mbox{ if } k=i \ \mbox{and} \ l\neq j, \\
c_{kj} +c_{ki}x, & \mbox{ if } k\neq i \ \mbox{and} \ l=j, \\
c_{ij} +(t_{ii}-t_{jj})x, & \mbox{ if } k=i \ \mbox{and} \ l=j.
\end{array}\right. $$
Let us define a right (odd) superderivation $D_{ij}$ that acts on the polynomial superalgebra $A(m|n)$ given by generators $c_{kl}$ for $1\leq k, l\leq m+n$ as follows.
$$c_{kl}D_{ij}=
\left\{\begin{array}{ll}
0, & \mbox{ if } k\neq i \ \mbox{and} \ l\neq j , \\
(-1)^{1+|c_{jl}|}c_{jl}, & \mbox{ if } k=i \ \mbox{and} \ l\neq j, \\
c_{ki}, & \mbox{ if } k\neq i \ \mbox{and} \ l=j, \\
t_{ii}-t_{jj}, & \mbox{ if } k=i \ \mbox{and} \ l=j.
\end{array}\right.$$ 
Remark \ref{actionofoddunipotent} implies that
$g(u)=u+(uD_{ij})x$ for every $u\in A(m|n)$. For example, $c^{\lambda}D_{ij}=(\lambda_i c^{\lambda-\epsilon_i}+\lambda_j c^{\lambda-\epsilon_j})c_{ji}$. Let $t^{\lambda}$ denote $\prod_{1\leq i\leq m+n}t_{ii}^{\lambda_i}$. Then
$t^{\lambda}D_{ij} =(\lambda_i t^{\lambda-\epsilon_i}+\lambda_j t^{\lambda-\epsilon_j})c_{ji}$ as well.

In what follows let $I_r=I_r(m|n)$ denote the subalgebra $h(Z_r)\subseteq Dist(T_r)$. Then $I=\bigcup_{r\geq 1} I_r=h(Z)$.

\section{The center of $Dist(GL(1|1))$}

Let $G=GL(1|1)$. When working with the $r$-th Frobenius kernel $G_r$ of $G$, for simplicity, denote $p^r$ by $q$ and $\Delta_{K[G_r]}$ just by $\Delta$.
The superalgebra $K[G_r]$ has a basis consisting of all elements 
$$c_{11}^{\lambda_1}c_{22}^{\lambda_2}c^a_{21}c^b_{12}=c^{\lambda}c^a_{21}c^b_{12} \mbox {, where } \lambda\in X^{(r)}(T) \mbox{ and } 0\leq a,b\leq 1 .$$ 

Our computations will be greatly simplified if we 
replace the previous basis of $K[G_r]$ by a new basis given by elements
$x^{\lambda}x_{21}^a x_{12}^b, \lambda\in X^{(r)}(T)$, where $x^{\lambda}=x_{11}^{\lambda_1}x_{22}^{\lambda_2}$, $a, b=0, 1$ and 
\[x_{11}=c_{11}, x_{12}=c_{11}^{-1}c_{12}, x_{21}=c_{11}^{-1}c_{21},
x_{22}=c_{22}-c_{11}^{-1}c_{21}c_{12}.\]
It is easy to see that 
\[x^{\lambda}=c^{\lambda}-\lambda_2 c^{\lambda-\epsilon_1-\epsilon_2}c_{21}c_{12} \mbox{ and } x^{\lambda}x_{21}x_{12}=c^{\lambda-2\epsilon_1}c_{21}c_{12}.\]

\begin{lm}\label{actionofsuperder}
The superderivations $D_{12}$ and $D_{21}$ act as follows.
\[
x^{\lambda}D_{12}=|\lambda|x^{\lambda}x_{21}, \ x_{21}D_{12}=0, \ x_{12}D_{12}=1-x_{11}^{q-1}x_{22},\]
\[x^{\lambda}D_{21}=|\lambda|x^{\lambda}x_{12}, \ x_{12}D_{21}=0 \mbox{ and } x_{21}D_{21}=x_{11}^{q-1}x_{22}-1.
\]
Consequently, \[(x^{\lambda}x_{21}x_{12})D_{12}=(1-x^{(q-1, 1)})x^{\lambda}x_{21} \mbox{ and } (x^{\lambda}x_{21}x_{12})D_{21}=(1-x^{(q-1, 1)})x^{\lambda}x_{12}.\]
\end{lm}
\begin{proof}
Straightforward calculations using that $x^{(-1,1)}=x_{11}^{-1}x_{22}=x^{(q-1,1)}$.
\end{proof}

Every element $f\in K[G_r]^G$ belongs to $K[G_r]^T$ and therefore has the form 
$\sum a_{\lambda}x^{\lambda} +\sum b_{\lambda}x^{\lambda}x_{21}x_{12}$, where $\lambda$ runs over $X^{(r)}(T)$.

If $f\in K[G_r]^G$, then $fD_{12}=0$ which gives 
\[
\sum_{\lambda\in X^{(r)}(T)}a_{\lambda}|\lambda|x^{\lambda}x_{21} +\sum_{\lambda\in X^{(r)}(T)} b_{\lambda}(1-x^{(q-1, 1)})x^{\lambda}x_{21}=0,
\]
or, equivalently 
\begin{equation}\label{eqo}
\sum_{\lambda\in X^{(r)}(T)}a_{\lambda}|\lambda|x^{\lambda} +\sum_{\lambda\in X^{(r)}(T)} b_{\lambda}(1-x^{(q-1, 1)})x^{\lambda}=0.
\end{equation}
Since the condition $fD_{21}=0$ also implies $(\ref{eqo})$, we conclude that the equation $(\ref{eqo})$  characterises elements of $K[G_r]^G$.

Define the representative $\overline{n}$ of the residue class of $n\in\mathbb{Z}$ modulo $q$ by $\overline{n}=n-q[\frac{n}{q}]$, and for 
$\lambda=(\lambda_1, \lambda_2)\in X(T)$ define $\overline{\lambda}\in X^{(r)}(T)$ by $\overline{\lambda}=(\overline{\lambda_1}, \overline{\lambda_2})$.
Denote $\epsilon_1-\epsilon_2$ by $\alpha$. The map $\lambda\mapsto\overline{\lambda+\alpha}$ is a bijection of $X^{(r)}(T)$ to itself and its inverse map is 
given by $\lambda\mapsto\overline{\lambda-\alpha}$. 
Each orbit $O_t$ of the map $\lambda\mapsto\overline{\lambda+\alpha}$ is uniquely defined by a parameter $t$ such that $0\leq t\leq q-1$, and it consists 
of all weights $\lambda\in X^{(r)}(T)$ such that $|\lambda|\equiv t\pmod q$. 

For each $\lambda\in X^{(r)}(T)$ define the element 
\[\gamma_{\lambda}=x^{\lambda}-x^{\overline{\lambda+\alpha}}+|\lambda|x^{\overline{\lambda+\alpha}}x_{21}x_{12},\]
and for $0\leq t \leq q-1$ denote $\sigma_t=\sum_{i=0}^{q-1} x^{\overline{(i,t-i)}} x_{21}x_{12}=\sum_{\lambda\in O_t}x^{\lambda} x_{21}x_{12}$.

\begin{tr}\label{generatorsinthesimplestcase}
The superspace $K[G_r]^G$ is generated by the following group of elements corresponding to parameters $t$ such that $0 \leq t\leq q-1$:

$\bullet$ the element $\sigma_t$ and the elements $x^{\lambda}$ for which $|\lambda|\equiv t\pmod q$, if $p|t$,

$\bullet$ the elements $\gamma_{\lambda}$ for which $|\lambda|\equiv t\pmod q$, if $p\not|t$.
\end{tr}
\begin{proof}
The equation $(\ref{eqo})$ is equivalent to the sytem of equations
\[
b_{\overline{\lambda+\alpha}}-b_{\lambda}=|\lambda|a_{\lambda} \mbox{ for } \lambda\in X^{(r)}(T).
\]

This system can be split into a collection of subsystems, each one of which corresponds to a unique orbit $O_t$, where $0\leq t\leq q-1$.

If $p|t$, then the corresponding subsystem implies that all $b_{\lambda}$ are equal to each other for $\lambda\in O_t$. 
Since $\sigma_{t}$ and $x^{\lambda}$ for $\lambda\in O_t$ can be easily checked to be invariants, they are generators corresponding to $t$.
 
If $p\not|t$, then 
\[\sum_{\lambda\in O_t}\frac{b_{\lambda}}{|\lambda|}\gamma_{\overline{\lambda-\alpha}}=
\sum_{\lambda\in O_t}\frac{b_{\overline{\lambda+\alpha}}-b_{\lambda}}{|\lambda|}x^{\lambda}+\sum_{\lambda\in O_t}
b_{\lambda}x^{\lambda}x_{21}x_{12}=\]
\[\sum_{\lambda\in O_t} a_{\lambda}x^{\lambda}+\sum_{\lambda\in O_t}
b_{\lambda}x^{\lambda}x_{21}x_{12}.
\]
Since all $\gamma_{\lambda}$, for $\lambda\in O_t$, are invariants, the theorem follows.
\end{proof}
We have the following formulas :
\[\begin{aligned}
\Delta(c^{\lambda})=&c^{\lambda}\otimes c^{\lambda}+\lambda_1 c^{\lambda-\epsilon_1}c_{12}\otimes c^{\lambda-\epsilon_1}c_{21}
+\lambda_2 c^{\lambda-\epsilon_2}c_{21}\otimes c^{\lambda-\epsilon_2}c_{12}\\
&-\lambda_1\lambda_2 c^{\lambda-\epsilon_1-\epsilon_2}c_{12}c_{21}\otimes c^{\lambda-\epsilon_1-\epsilon_2}c_{21}c_{12},
\end{aligned}\]
\[\begin{aligned}
\Delta(c^{\lambda}c_{12})=&c^{\lambda+\epsilon_1}\otimes c^{\lambda}c_{12}+c^{\lambda}c_{12}\otimes c^{\lambda+\epsilon_2}+
\lambda_1 c^{\lambda}c_{12}\otimes c^{\lambda-\epsilon_1}c_{21}c_{12} \\
&-\lambda_2 c^{\lambda-\epsilon_2}c_{21}c_{12}\otimes c^{\lambda}c_{12},
\end{aligned}\]
\[\begin{aligned}
\Delta(c^{\lambda}c_{21})=&c^{\lambda}c_{21}\otimes c^{\lambda+\epsilon_1} +c^{\lambda+\epsilon_2}\otimes c^{\lambda}c_{21}+
\lambda_1 c^{\lambda-\epsilon_1}c_{21}c_{12}\otimes c^{\lambda}c_{21}\\
&-\lambda_2 c^{\lambda}c_{21}\otimes c^{\lambda-\epsilon_2}c_{21}c_{12},
\end{aligned}\]
\[\begin{aligned}
\Delta(c^{\lambda}c_{21}c_{12})=&c^{\lambda+\epsilon_1}c_{21}\otimes c^{\lambda+\epsilon_1}c_{12}-c^{\lambda+\epsilon_2}c_{12}\otimes c^{\lambda+\epsilon_2}c_{21}
+c^{\lambda}c_{21}c_{12}\otimes c^{\lambda+\epsilon_1+\epsilon_2}\\
&+c^{\lambda+\epsilon_1+\epsilon_2}\otimes c^{\lambda}c_{21}c_{12}+(\lambda_1-\lambda_2)c^{\lambda}c_{21}c_{12}\otimes c^{\lambda}c_{21}c_{12}
\end{aligned}\]

\begin{lm}\label{pairingwithc-s}
For every $\lambda\in X^{(r)}(T)$ we have
\[
\left(\begin{array}{c}
e \\
\pi
\end{array}\right)e_{21}e_{12}(c^{\lambda}c_{21}c_{12})=-\left(\begin{array}{c}
\lambda+\epsilon_1+\epsilon_2 \\
\pi
\end{array}\right)
\]
and 
\[
\left(\begin{array}{c}
e \\
\pi
\end{array}\right)e_{21}e_{12}(c^{\lambda})=-\lambda_2 \left(\begin{array}{c}
\lambda \\
\pi
\end{array}\right) .
\]
\end{lm}
\begin{proof}
To prove the first equation, using the above formulas for comultiplication, compute
\[\begin{aligned}
&\left(\begin{array}{c}
e \\
\pi
\end{array}\right)e_{21}e_{12}(c^{\lambda}c_{21}c_{12})=(\left(\begin{array}{c}
e \\
\pi
\end{array}\right)e_{21}\otimes e_{12})\Delta(c^{\lambda}c_{21}c_{12})=\\
&-\left(\begin{array}{c}
e \\
\pi
\end{array}\right)e_{21}(c^{\lambda+\epsilon_1}c_{21})=-(\left(\begin{array}{c}
e \\
\pi
\end{array}\right)\otimes e_{21})\Delta(c^{\lambda+\epsilon_1}c_{21})=-\left(\begin{array}{c}
\lambda+\epsilon_1+\epsilon_2 \\
\pi
\end{array}\right).
\end{aligned}\]
The proof of the second equation is analogous.
\end{proof}
\begin{lm}\label{rightpairing}
For every $\lambda\in X^{(r)}(T)$ we have
\[\left(\begin{array}{c}
e \\
\pi
\end{array}\right)(x^{\lambda})=\left(\begin{array}{c}
\lambda \\
\pi
\end{array}\right),
\
\left(\begin{array}{c}
e \\
\pi
\end{array}\right)(x^{\lambda}x_{21}x_{12})=0, \]
\[\left(\begin{array}{c}
e \\
\pi
\end{array}\right)e_{21}e_{12}(x^{\lambda})=0 \mbox{ and }
\left(\begin{array}{c}
e \\
\pi
\end{array}\right)e_{21}e_{12}(x^{\lambda}x_{21}x_{12})=-\left(\begin{array}{c}
\lambda-\epsilon_1+\epsilon_2 \\
\pi
\end{array}\right).
\]
\end{lm}
\begin{proof}
The proof follows from 
$\left(\begin{array}{c}
e \\
\pi
\end{array}\right)(c^{\lambda})=\left(\begin{array}{c}
\lambda \\
\pi
\end{array}\right)$,
$\left(\begin{array}{c}
e \\
\pi
\end{array}\right)(c^{\lambda}c_{21}c_{12})=0
$
and Lemma \ref{pairingwithc-s}.
\end{proof}
The map $e_{11}\mapsto e'_{11}=e_{11}+1, e_{22}\mapsto e'_{22}=e_{22}-1$ can be extended to an automorphism of algebra $Dist(T)$ which maps
$\left(\begin{array}{c}
e \\
\pi
\end{array}\right)$ to 
$\left(\begin{array}{c}
e' \\
\pi
\end{array}\right)=
\left(\begin{array}{c}
e +\epsilon_1-\epsilon_2\\
\pi
\end{array}\right)$.
\begin{lm}\label{asimplification}
For every $\lambda\in X^{(r)}(T)$ we have 
\[\left(\begin{array}{c}
e' \\
\pi
\end{array}\right)e_{21}e_{12}(x^{\lambda})=0 \mbox{ and } \left(\begin{array}{c}
e' \\
\pi
\end{array}\right)e_{21}e_{12}(x^{\lambda}x_{21}x_{12})=-\left(\begin{array}{c}
\lambda \\
\pi
\end{array}\right).
\]
\end{lm}
\begin{proof}
Using the formal identity 
$$
\sum_{0\leq i\leq k}(-1)^{k-i}\left(\begin{array}{c}
x \\
i
\end{array}\right)=\left(\begin{array}{c}
x-1 \\
k
\end{array}\right)
$$
that can be derived from the identites on p. 195 of \cite{cartlust},
we obtain
$$
\left(\begin{array}{c}
e' \\
\pi
\end{array}\right)=\sum_{\substack{\pi_1-1\leq\beta_1\leq\pi_1,\\ 0\leq\beta_2\leq\pi_2 }}(-1)^{\pi_2-\beta_2}\left(\begin{array}{c}
e \\
\beta
\end{array}\right).
$$
Using Lemma \ref{rightpairing} and the above formal identity, applied to integers, we conclude the proof.  
\end{proof}
Let $g_{\pi}$ denote the element
\[
\sum_{\pi\preceq\beta\preceq (q-1, q-1)}(-1)^{|\beta|}\left(\begin{array}{c}
\beta \\
\pi
\end{array}\right)\left(\begin{array}{c}
e' \\
\beta
\end{array}\right)e_{21}e_{12},
\]
and let $h_{\pi}$ denote the element
\[
\sum_{\pi\preceq\beta\preceq (q-1, q-1)}(-1)^{|\beta|}\left(\begin{array}{c}
\beta \\
\pi
\end{array}\right)\left(\begin{array}{c}
e \\
\beta
\end{array}\right),
\]
where $\pi\in X^{(r)}(T)$. Observe that $h_{(0, 0)}=\Delta^{(r)}_T$.
\begin{lm}\label{abinomialsum}
For any non-negative integers $t$ and $l$ such that $t\leq l$ we have
$$\sum_{t\leq i\leq l}(-1)^i\left(\begin{array}{c}
l\\
i
\end{array}\right)\left(\begin{array}{c}
i\\
t
\end{array}\right)=(-1)^t\delta_{tl}.
$$
\end{lm}
\begin{proof}
The statement follows from
\[\begin{aligned}
&\sum_{t\leq i\leq l}(-1)^i\left(\begin{array}{c}
l\\
i
\end{array}\right)\left(\begin{array}{c}
i\\
t
\end{array}\right)=\sum_{t\leq i\leq l}(-1)^i\frac{l!}{t!(l-i)!(i-t)!}\\
&=
(-1)^t\left(\begin{array}{c}
l\\
t
\end{array}\right)\sum_{0\leq j\leq s}(-1)^s\left(\begin{array}{c}
s\\
j
\end{array}\right),
\end{aligned}
\]
where $s=l-t, j=i-t$.
\end{proof}
\begin{pr}\label{dualbases}
For every $\pi, \lambda\in X^{(r)}(T)$ we have 
\[g_{\pi}(x^{\lambda}x_{21}x_{12})=(-1)^{|\pi|+1}\delta_{\pi, \lambda} \mbox{ and } g_{\pi}(x^{\lambda})=0.\]
Additionally, 
\[h_{\pi}(x^{\lambda})=(-1)^{|\pi|}\delta_{\pi, \lambda} \mbox{ and } h_{\pi}(x^{\lambda}x_{21}x_{12})=0.\]
\end{pr}
\begin{proof}
Combining Lemmas \ref{asimplification} and \ref{abinomialsum} we derive
\[
g_{\pi}(x^{\lambda}x_{21}x_{12})=
-\sum_{\pi\preceq\beta\preceq\lambda}(-1)^{|\beta|}\left(\begin{array}{c}
\beta \\
\pi
\end{array}\right)\left(\begin{array}{c}
\lambda \\
\beta
\end{array}\right)=(-1)^{|\pi|+1}\delta_{\pi, \lambda}
\mbox{ and }  g_{\pi}(x^{\lambda})=0.
\]
To prove the second statement, one has to combine Lemmas \ref{rightpairing} and \ref{abinomialsum}.
\end{proof}
\begin{lm}\label{newformofintegral}
We can identify the integral $\nu_r$ as 
\[
\nu_r=\sum_{\substack{0\leq\beta_1\leq q-1, \\ \beta_2 =q-1}}(-1)^{|\beta|+1}\beta_1
\left(\begin{array}{c}
e' \\
\beta
\end{array}\right)e_{21}e_{12}=-g_{1, q-1}.
\]
\end{lm}
\begin{proof}
Using the formula
\[
\left(\begin{array}{c}
e \\
\pi
\end{array}\right)=
\sum_{\substack{0\leq\beta_1\leq\pi_1, \\ \pi_2 -1\leq\beta_2\leq\pi_2}}
(-1)^{\pi_1-\beta_1}\left(\begin{array}{c}
e' \\
\beta
\end{array}\right)
\]
that is symmetric to the formula from the proof of Lemma \ref{asimplification}, we obtain
\[\begin{aligned}
\Delta^{(r)}_T &=\sum_{0\preceq\pi\preceq (q-1, q-1)}(-1)^{|\pi|}\sum_{0\leq\beta_1\leq\pi_1 , \pi_2-1\leq\beta_2\leq\pi_2}
(-1)^{\pi_1-\beta_1}\left(\begin{array}{c}
e' \\
\beta
\end{array}\right)\\
&=\sum_{0\leq\beta_1\leq q-1, \ \beta_2 =q-1}(-1)^{|\beta|}(q-\beta_1)
\left(\begin{array}{c}
e' \\
\beta
\end{array}\right)=-g_{1,q-1}.
\end{aligned}
\]
\end{proof}

\begin{lm}\label{condemnedlemma}
We have the following identities.
\[\nu_r x^{\pi}=
(-1)^{|\overline{\alpha-\pi}|}g_{\overline{\alpha-\pi}}\]
and 
\[
\nu_r x^{\pi}x_{21}x_{12}=
(-1)^{|\overline{\alpha-\pi}|+1}h_{\overline{\alpha-\pi}}\]
\end{lm} 
\begin{proof}
The first statement follows from Proposition \ref{dualbases} and
\[(\nu_r x^{\pi})(x^{\lambda}x_{21}x_{12})=-g_{1, q-1}(x^{\pi+\lambda}x_{21}x_{12})=-\delta^{(q)}_{1, \pi_1+\lambda_1}\delta^{(q)}_{q-1, \pi_2+\lambda_2},\]
and the second statement follows from Proposition \ref{dualbases} and 
\[(\nu_r x^{\pi}x_{21}x_{12})(x^{\lambda})=-g_{1, q-1}(x^{\pi+\lambda}x_{21}x_{12})=-\delta^{(q)}_{1, \pi_1+\lambda_1}\delta^{(q)}_{q-1, \pi_2+\lambda_2}.\]
\end{proof}

The center $Z=\cup_{r>1} Z_r$ of $Dist(G)$ is described in the following theorem.
\begin{tr}\label{center}
The generators of $Z_r=Dist(G_r)^G$ correspond to orbits $O_t$ for $0\leq t\leq q-1$.

If $p| t$, then the generators are all elements $g_{\lambda}$ for $\lambda\in O_t$ 
together with one additional element $\sum_{\lambda\in O_t} (-1)^{|\lambda|}h_{\lambda}$.

If $p\not| t$, then the generators are
$(-1)^{|\lambda|+|\overline{\lambda+\alpha}|}g_{\overline{\lambda+\alpha}}-g_{\lambda}+|\lambda|h_{\lambda}$ 
for each $\lambda\in O_t$. 
\end{tr} 
\begin{proof} It follows from Corollary \ref{corko} , Theorem \ref{generatorsinthesimplestcase} and Lemma \ref{condemnedlemma}.
\end{proof}

\begin{tr}\label{harish-chandrapolynomials}
The space $I_r$ is generated by the following elements corresponding to orbits $O_t$ for
$0\leq t\leq q-1$.

If $p| t$, then we have only one generator $\sum_{\lambda\in O_t} (-1)^{|\lambda|}h_{\lambda}$.

If $p\not| t$, then the generators are all elements $h_{\lambda}$ for $\lambda\in O_t$.
\end{tr}
\begin{proof}
Recall that every summand of a central element that is a multiple of $x_{21}x_{12}$ is mapped to zero by the Harish-Chandra homomorphism.
Now, the proof follows from Theorem \ref{center}.
\end{proof}

\section{Blocks}

Assume that $\lambda$ is a dominant weight of $G$ and recall $\alpha=\epsilon_1-\epsilon_2$.

Define the {\it usual block} $B(\lambda)$ of $G$ as a set of weights $\mu\in X(T)^+$ such that there is a sequence 
$\lambda=\lambda_1$, \ldots, $\lambda_i$, \ldots, $\lambda_r=\mu$ satisfying $Ext^1_G(L(\lambda_i), L(\lambda_{i+1}))\neq 0$ or $Ext^1_G(L(\lambda_{i+1}), L(\lambda_{i}))\neq 0$
for each $i=1, \ldots, r-1$. 

Define the {\it Harish-Chandra block} $HC(\lambda)$ as a set of weights $\mu\in X(T)^+$ such that $z_{\lambda}=z_{\mu}$ for every central element $z$. 
As we have seen earlier, this condition is equivalent to $hv_{\lambda}=hv_{\mu}$ for every $h\in I$. 

Recall a variant of a {\it linkage principle} from \cite{kuj2}. 
Let $P=\oplus_{r\in\mathbb{Z}/p\mathbb{Z}}\mathbb{Z}\Lambda_r\oplus\mathbb{Z}\delta$ be a weight lattice of affine Lie algebra $\hat{\mathfrak{sl}}_p(\mathbb{C})$ and 
let $\mbox{wt} : X(T)\to P$ be a function related to the {\it crystal structure} on $X(T)$. 
Define the {\it Kujawa block} $K(\lambda)$ as a set of weights
$\mu\in X(T)^+$ such that $\mbox{wt}(\lambda)=\mbox{wt}(\mu)$. 

It is natural to ask about the relationship between $B(\lambda)$, $HC(\lambda)$, and $K(\lambda)$ for $\lambda\in X(T)^+$.
Clearly, $B(\lambda)\subseteq HC(\lambda)$ for every $\lambda\in X(T)^+$. Since every extension of $L(\mu)$ by $L(\lambda)$ is finite-dimensional, hence integrable, we have an isomorphism  $Ext^1_G(L(\lambda),L(\mu))\simeq Ext^1_{Dist(G)}(L(\lambda), L(\mu))$.
Then Theorem 3.7 of \cite{kuj2} implies that, for every $\lambda, \mu\in X(T)^+$, $Ext^1_G(L(\lambda), L(\mu))\neq 0$ implies $\mbox{wt}(\lambda)=\mbox{wt}(\mu)$.
Therefore $B(\lambda)\subseteq K(\lambda)$ for every $\lambda\in X(T)^+$.

We will now investigate the relationship between $B(\lambda)$, $HC(\lambda)$ and $K(\lambda)$ further in the simplest case when $G=GL(1|1)$.
From now on assume $G=GL(1|1)$.

\begin{lm}\label{justblocks}
If $p| |\lambda|$, then $B(\lambda)=\lambda +\mathbb{Z}\alpha$. If $p\not||\lambda|$, then $B(\lambda)=\{\lambda\}$.
\end{lm}
\begin{proof}
Let $p\not||\lambda|$. If $Ext^1_G(L(\mu), L(\lambda))\neq 0$, then there is a $G$-supermodule $M$ such that
$L(\lambda)$ is a socle of $M$ and $M/L(\lambda)\simeq L(\mu)$. Thus $M$ can be embedded into an injective envelope $I(\lambda)$ of
$L(\lambda)$. By Lemma 7.1 (c) of \cite{marzub}, we infer that $I(\lambda)=L(\lambda)$. By Chevalley duality, if $Ext^1_G(L(\lambda), L(\mu)\neq 0$, then 
$Ext^1_G(L(\mu), L(\lambda))\neq 0$ and we conclude that $B(\lambda)=\{\lambda\}$.

Finally, assume $p||\lambda|$.  
If $Ext^1_G(L(\lambda), L(\mu))\neq 0$ or $Ext^1_G(L(\mu), L(\lambda))\neq 0$, then by Lemma 7.1 (c) of \cite{marzub}
we obtain that $\lambda=\mu\pm \alpha$. Conversely, if $\lambda=\mu\pm \alpha$, then 
$Ext^1_G(L(\lambda), L(\mu))\neq 0$ and $Ext^1_G(L(\mu), L(\lambda))\neq 0$ using Lemma 7.1 (c) of \cite{marzub} again.
Therefore $B(\lambda)= \lambda+\mathbb{Z}\alpha$. 
\end{proof}

For the description of the usual blocks in Schur superalgebra $S(1|1)$ see also Proposition 2.1 of \cite{marzub0}.

For every $a\in\mathbb{Z}$ write $a=pd +s$ for $1\leq s\leq p$ and define $\gamma_a=\Lambda_s-\Lambda_{s-1}-d\delta$. 
If  $\lambda=(\lambda_1, \lambda_2)$, then $\mbox{wt}(\lambda)=\gamma_{\lambda_1}-\gamma_{-\lambda_2}$.
In particular, if $\lambda_1=pd_1+s_1, -\lambda_2=pd_2+s_2$, then
\[\mbox{wt}(\lambda)=\Lambda_{s_1}-\Lambda_{s_1-1}-\Lambda_{s_2}+\Lambda_{s_2-1}-(d_1-d_2)\delta .\]
This implies that if $\mu=\lambda +pt\alpha$, where $t\in\mathbb{Z}$, then $\mbox{wt}(\lambda)=\mbox{wt}(\mu)$. 

\begin{lm}\label{pdivides}
If $p| |\lambda|$, then $K(\lambda)=\lambda +\mathbb{Z}\alpha$. If $p\not||\lambda|$, then
$K(\lambda)=\lambda +p\mathbb{Z}\alpha$. 
\end{lm}
\begin{proof}
First observe that $p||\lambda|$ if and only if $s_1=s_2$ and if and only if $\mbox{wt}(\lambda)\in\mathbb{Z}\delta$. Therefore, if $\mu\in K(\lambda)$ and
$\mu_1=pd'_1 +s'_1, -\mu_2=pd'_2 +s'_2$, then $s'_1=s'_2$ and $d_1-d_2=d'_1 -d'_2$. Thus $\mu-\lambda\in \mathbb{Z}\alpha$ and 
$K(\lambda)\subseteq \lambda +\mathbb{Z}\alpha=B(\lambda)$. The opposite inclusion $B(\lambda)\subseteq K(\lambda)$ was already noted before.

Assume that $p\not||\lambda|$. We only need to show that $K(\lambda)\subseteq \lambda +p\mathbb{Z}\alpha$ since the opposite inclusion is obvious.
Using previous notation, this is equivalent to showing that $s_1=s'_1$, which implies $s_2=s'_2$. 

Assume that $s_1\neq s'_1$. Since $s_1\neq s_2$ and $s'_1\neq s'_2$,
the equality $\mbox{wt}(\lambda)=\mbox{wt}(\mu)$ implies that $\Lambda_{s_1}=\Lambda_{s'_2 -1}$ and 
$\Lambda_{s_2 -1}=\Lambda_{s'_1}$. 
Consequently,  $s_1\equiv s'_2-1 \pmod p$ and $s_2-1\equiv s'_1 \pmod p$.

Therefore either $s_1=p$ and $s'_2=1$, or $1\leq s_1=s'_2-1<p$, and analogously either $s'_1=p$ and $s_2=1$, or $1\leq s'_1=s_2-1<p$. 
This gives three possible cases to check. 
If $1\leq s_1=s'_2-1<p$ and $s'_1=p, s_2=1$, then 
\[\Lambda_{s_1}-\Lambda_{s_1 -1}-\Lambda_1+\Lambda_p=\Lambda_p-\Lambda_{p-1}-\Lambda_{s_1 +1}+\Lambda_{s_1},\]
which is a contradiction. The remaining two cases are left for the reader to verify.
\end{proof}
\begin{rem}\label{howtocompute}
For every $\lambda\in X(T)^+$ and $r\geq 1$ there are isomorphisms
$L(\lambda)|_{G_r}\simeq
L_r(\lambda)\simeq L_r(\overline{\lambda})\simeq L(\overline{\lambda})|_{G_r}$. Thus $z_{\lambda}=z_{\overline{\lambda}}$ for arbitrary
$z\in Z_r$. The same statement is true for $G=GL(m|n)$ and any $m$ and $n$, provided $r$ is sufficiently large $r$; say
$q=p^r > 2|\lambda_i|$ for every $1\leq i\leq m+n$.
\end{rem}
\begin{lm}\label{hcblocks}
If $p| |\lambda|$, then $HC(\lambda)=\lambda +\mathbb{Z}\alpha$. If $p\not||\lambda|$, then $HC(\lambda)=\{\lambda\}$.
\end{lm}
\begin{proof}
Fix a pair of weights $\mu \neq \mu' \in X(T)^+$. We will determine if they belong to the same Harish-Chandra block by comparing values of $hv_{\mu}$ and $hv_{\mu'}$ for 
generating polynomials $h\in I=\cup I_r$ listed in Theorem \ref{harish-chandrapolynomials}. 
We can assume that $|\mu|=|\mu'|=|\lambda|$, or that $\mu, \mu'\in \lambda +\mathbb{Z}\alpha$, since otherwise weights $\mu$ and $\mu'$ belong to different Harish-Chandra blocks.
Using Remark \ref{howtocompute}, it is enough to consider sufficiently large $r$.
Take $r$ large enough so that $q=p^r > |\mu_1|, |\mu_2|, |\mu'_1|, |\mu'_2|$. Then $|\mu|, |\mu'|\leq 2q-2$.

Assume that $p\not||\lambda|$. Using Theorem \ref{harish-chandrapolynomials} and the fact that 
$h_{\pi}v_{\overline{\mu}}=(-1)^{|\pi|}\delta_{\pi, \overline{\mu}}v_{\overline{\mu}}$ for each $\pi\in X^{(r)}(T)$, we see immediately that
$\mu$ and $\mu'$ belong to different Harish-Chandra blocks. Hence $HC(\mu)=\{\mu\}$.

If $p||\lambda|$, then every element $\sum_{\pi\in O_{t}}(-1)^{|\pi|} h_{\pi}$ for $1\leq t\leq q-1$  and $p|t$ attains the same value (either 0 or 1) on both $v_{\overline{\mu}}$  and $v_{\overline{\mu'}}$. Thus
$HC(\lambda)=\lambda +\mathbb{Z}\alpha$.
\end{proof}

We have shown that for $G=GL(1|1)$ we have $HC(\lambda)=B(\lambda)$ for every $\lambda$, that is Harish-Chandra polynomials determine usual blocks, analogous to the classical case of general linear groups. However, $B(\lambda)\subsetneq K(\lambda)$ if $p\not| |\lambda|$ showing that Kujawa blocks are bigger than usual blocks in this case.

\end{document}